\documentclass[11pt,a4paper]{article}

\usepackage[T1]{fontenc}
\usepackage[utf8]{inputenc}
\usepackage{lmodern}
\usepackage[margin=2.5cm]{geometry}
\usepackage{microtype}
\usepackage{amsmath,amsthm,amsfonts,amssymb,mathtools}
\usepackage[numbers,sort&compress]{natbib}
\usepackage{graphicx}
\usepackage{enumerate}
\usepackage{authblk}
\usepackage{hyperref}
\hypersetup{
  colorlinks=true,
  linkcolor=blue,
  citecolor=blue,
  urlcolor=blue,
  pdftitle={Conformal robust set estimation},
  pdfauthor={Alejandro Cholaquidis, Emilien Joly, Leonardo Moreno},
  pdfsubject={Robust conformal prediction and central set estimation},
  pdfkeywords={Conformal prediction, robust set estimation, half-mass radius,
    central regions, nearest neighbours, contamination, Hausdorff distance}
}
\numberwithin{equation}{section}

\theoremstyle{plain}
\newtheorem{thm}{Theorem}[section]
\newtheorem{lem}[thm]{Lemma}
\newtheorem{prop}[thm]{Proposition}
\newtheorem{cor}[thm]{Corollary}

\theoremstyle{definition}
\newtheorem{definition}[thm]{Definition}
\newtheorem{assumption}{Assumption}

\theoremstyle{remark}
\newtheorem{rem}[thm]{Remark}

\newcommand{\R}{\mathbb{R}}

\renewcommand{\hat}{\widehat}
\renewcommand{\tilde}{\widetilde}
\newcommand{\dP}{\delta_P}
\newcommand{\Qba}{Q_{\beta_\alpha}}

\title{Conformal robust set estimation}
\author[1]{Alejandro Cholaquidis}
\author[2]{Emilien Joly}
\author[3]{Leonardo Moreno}
\affil[1]{Centro de Matem\'atica, Facultad de Ciencias,
Universidad de la Rep\'ublica, Uruguay}
\affil[2]{Centro de Investigaci\'on en Matem\'aticas (CIMAT), M\'exico}
\affil[3]{Departamento de M\'etodos Cuantitativos, Facultad de Ciencias
Econ\'omicas y de Administraci\'on, Universidad de la Rep\'ublica, Uruguay}

\date{}

\begin{document}
\maketitle

\begin{abstract}
We introduce a robust conformal prediction method based on the half-mass
radius $\dP(x):=\inf\{r>0:P(B(x,r))>1/2\}$. Its empirical version is
exactly the distance from a point to its $\mathbf{k}$-th nearest
neighbour, with $\mathbf{k}=\lfloor n/2\rfloor+1$. The resulting conformal
sets retain finite-sample marginal validity and asymptotically target the
robust population central set $\Qba:=\{x\in\R^d:\dP(x)\le\beta_\alpha\}$.
We prove consistency in symmetric-difference probability, derive
non-asymptotic exponential bounds and obtain the rate
$O((\log n/n)^{1/2})$ under a local mass condition. Under the same local
regularity on the uncontaminated law, we also establish quantitative
$O(\eta)$ stability, uniformly over arbitrary contaminations of
sufficiently small mass $\eta$, prove Hausdorff convergence of the
empirical central sets, and provide computable ball-based representations
using only pairwise distances.
\end{abstract}



\section{Introduction}

Conformal prediction is a general methodology for constructing prediction
sets with finite-sample, distribution-free coverage under exchangeability
\citep{vovk03,fontana2023}. Given a sample $\aleph_n=(X_1,\dots,X_n)$ and a
new observation $X_{n+1}$, conformal methods build, for $\alpha\in(0,1)$, a
data-dependent set $\gamma^\alpha(\aleph_n)$ such that
$P(X_{n+1}\in\gamma^\alpha(\aleph_n))\ge1-\alpha$ for every exchangeable
distribution of $(X_1,\dots,X_{n+1})$. This validity does not rely on
parametric or smoothness assumptions: conformal prediction converts a
user-chosen notion of atypicality into a calibrated set-valued predictor,
thereby separating coverage from modelling assumptions.

Originally developed for classification \citep{vovk98,vovk03}, conformal
prediction has since been extended to regression and conditional coverage
\citep{waser14,waser14_2,kuleshov2018conformal,romano2019conformalized},
Bayesian and pseudo-Bayesian procedures \citep{fong2021conformal}, and
functional and time-series data \citep{lei2015conformal,fontana20,diqui22};
see \citep{bala14,fontana2023} for overviews. In regression, non-conformity
scores are typically built from residuals, possibly after fitting a
regression model, and prediction sets are obtained by thresholding these
scores. When the score instead reflects the geometry of the underlying
distribution, the resulting prediction region can be both informative and
stable.

More recently, several works have investigated conformal inference beyond
the ideal exchangeable setting, including distribution drift, online
distribution shifts, distributionally robust validation and noisy labels
\citep{barber2023conformal,cauchois2024robust,gibbs2024conformal,einbinder2024label}.
These contributions show that robustness in conformal inference can be
pursued by modifying the calibration mechanism, reweighting conformity
scores, or explicitly modelling perturbations of the data-generating
distribution.

Comparatively less attention has been paid to the geometric structure and
robustness of conformal prediction regions. In the presence of outliers,
heavy tails, or multimodal structure, residual-based scores may produce
sets that are unnecessarily large, sensitive to a small fraction of
atypical observations, or poorly aligned with the distributional geometry.
By contrast, robust geometric functionals such as distances to measures
and depth-based central sets have been studied extensively in geometric
and topological inference, where they are known to be stable under
perturbations and insensitive to small contaminations
\citep{chazal2011geometric,fasy2018robust,cholaquidis2024gros}. This
motivates conformal scores based on local mass rather than on global
moments.

A related line of work uses conformal ideas for outlier and
out-of-distribution detection. Conformal $p$-values have been used to test
whether future observations are compatible with a reference distribution
\citep{bates2023testing}, while prediction sets in classification have
been combined with outlier detection by allowing the conformal set to be
empty for atypical covariates \citep{guan2022prediction}. More recently,
integrative conformal $p$-values have been proposed to exploit labelled
outliers in out-of-distribution testing \citep{liang2024integrative}.
These works focus primarily on testing or detecting atypical observations.
Our goal is complementary: we study conformal prediction regions as robust
geometric central sets and analyse their population target, stability and
convergence.

In this work we introduce a conformal procedure based on a
distance-to-a-measure type functional. Working in $\R^d$, we take as
non-conformity score the half-mass radius around a point $z$: the smallest
radius of a ball centred at $z$ that contains a strict majority of the
sample. As we show below, this is exactly the distance from $z$ to its
$\mathbf{k}$-th nearest neighbour, with
$\mathbf{k}:=\lfloor n/2\rfloor+1$. At the population level, the
corresponding mass-radius functional is
$\dP(x):=\inf\{r>0:P(B(x,r))>1/2\}$, which is closely related to the
distance-to-a-measure framework of \cite{chazal2011geometric}. Its
stability under perturbations and its robustness to small fractions of
outliers make it a natural non-conformity score for robust conformal
prediction.

\paragraph{Our contribution.}
We analyse the conformal procedure induced by the half-mass radius score
and relate its prediction regions to the robust central sets defined by
the level sets of $\dP$. Our main contributions are the following.
\begin{itemize}
\item We introduce the non-conformity score
\[
A(\mathcal{B},z):=\min\{\max_{j\in I}\|z-y_j\|:I\subset\{1,\dots,n\},\
|I|>n/2\}
\]
of a point $z$ with respect to a bag
$\mathcal{B}=(y_1,\dots,y_n)$, and prove that it equals the distance from
$z$ to its $\mathbf{k}$-th nearest neighbour in $\mathcal{B}$, with
distances counted with multiplicities (Lemma~\ref{prop:minmax=knn},
Corollary~\ref{cor:A-empirical-radius}).

\item We show that the resulting conformal prediction set
$\gamma^\alpha(\aleph_n)$ asymptotically targets the robust population
central set $\Qba$, where $Q_\beta:=\{x\in\R^d:\dP(x)\le\beta\}$ and
$\beta_\alpha:=\inf\{\beta\ge0:P(Q_\beta)\ge1-\alpha\}$
(Definition~\ref{def:central}). By right-continuity of the distribution
function of $\dP(X)$, $X\sim P$, we have $P(\Qba)\ge1-\alpha$, with
equality whenever $P(\dP(X)=\beta_\alpha)=0$. Under the regularity
assumptions stated below, we prove that
$P(X_{n+1}\in\gamma^\alpha(\aleph_n)\,\triangle\,\Qba)\to0$.

\item We derive non-asymptotic exponential bounds for the probability of
the symmetric difference between $\gamma^\alpha(\aleph_n)$ and $\Qba$.
Under a local two-sided mass condition at the level $\beta_\alpha$, these
bounds yield the explicit rate $O((\log n/n)^{1/2})$, with a constant that
depends on the ambient dimension
(Theorem~\ref{thm:consistency-corrected-omega},
Corollary~\ref{cor:rate-symmetric-difference}).

\item We establish quantitative population robustness properties. Under
arbitrary contamination of mass $\eta<1/2$, the half-mass radius is
bracketed by mass radii of the clean law
(Proposition~\ref{prop:contamination-bracketing}); for sufficiently small
$\eta$, and under the same local regularity condition used in our
consistency analysis, it is perturbed by at most $O(\eta)$ on the relevant
level strip, with an explicit constant
(Lemma~\ref{cor:contamination-local-stability}). Both the calibrated level
$\beta_\alpha$ and the calibrated central set are stable to order
$O(\eta)$, under assumptions on the clean law alone
(Lemma~\ref{lem:quantile-stability},
Theorem~\ref{cor:calibrated-set-stability}). Moreover, the miscoverage
level with respect to the clean law is inflated at most by the factor
$(1-\eta)^{-1}$, without any distributional assumptions
(Remark~\ref{rem:clean-coverage}); see also
Remark~\ref{rem:end-to-end-sampling} for conformal regions built from
contaminated samples.

\item We give geometric and computational representations of the empirical
central sets $\hat Q_\beta:=\{z\in\R^d:\delta_{P^{\aleph_n}}(z)\le\beta\}$,
where $P^{\aleph_n}$ is the empirical measure. We express $\hat Q_\beta$
through unions and intersections of Euclidean balls
(Propositions~\ref{prop:Qhat-ball-representation}
and~\ref{prop:Sbeta-dual-representation}) and construct conservative inner
proxies based on local radii $\mathcal{D}_i$, computable from pairwise
distances (Proposition~\ref{prop:conservative-proxy}). Under
Assumption~\ref{H1} and the level-set continuity Assumption~\ref{H3}, for
which we provide a geometric characterization
(Proposition~\ref{prop:H3-characterization}), we also prove Hausdorff
convergence of $\hat Q_{\beta_\alpha}$ to $\Qba$
(Proposition~\ref{prop:Qhat-H-conv}).
\end{itemize}

Taken together, these results show that conformal prediction based on
half-mass radii is marginally valid in finite samples and, at the same
time, converges to a robust geometric target determined by the underlying
distribution. This provides a principled way to integrate robust geometric
inference into the conformal prediction framework.

\paragraph{Organization.}
Section~\ref{confinf} reviews conformal inference and fixes notation.
Section~\ref{robustconf} introduces the half-mass radius score and
establishes its $\mathbf{k}$-NN representation.
Section~\ref{population-target} defines the population target $\Qba$ and
states the standing assumptions. Section~\ref{contamination} records the
population robustness of the half-mass radius under contamination,
including the stability of the calibrated level and of the calibrated
central set, and a distribution-free transfer of the coverage guarantee to
the uncontaminated law. Section~\ref{consistency} proves the consistency
and concentration results linking the conformal region to $\Qba$,
including the symmetric-difference rate and the Hausdorff convergence of
the empirical central sets. Section~\ref{comp} gives computational
representations and conservative inner approximations of the empirical
central sets. Section~\ref{experiments} reports numerical experiments, and
Appendix~\ref{app} collects technical results.

\section{Conformal inference}\label{confinf}

We briefly recall the full (transductive) conformal construction; see
\cite{vovk03,fontana2023} for general background. Let
$(\mathcal{X},d)$ be a metric space, which the reader may take to be $\R^d$
with the Euclidean distance from Section~\ref{consistency} onwards. A
\emph{non-conformity score} is a measurable map $A(\mathcal{B},z)$ taking a
finite bag (multiset) $\mathcal{B}=(y_1,\dots,y_n)$ of points of
$\mathcal{X}$ and a point $z\in\mathcal{X}$, and returning a real number
that quantifies how atypical $z$ is with respect to $\mathcal{B}$. The
score is required to be symmetric in $\mathcal{B}$, that is, invariant
under permutations of $(y_1,\dots,y_n)$.

Given the sample $\aleph_n=(X_1,\dots,X_n)$ and a candidate point
$z\in\mathcal{X}$, write $X_{n+1}:=z$ and define the leave-one-out bags
$\mathcal{B}_i(z):=(X_1,\dots,X_{i-1},X_{i+1},\dots,X_{n+1})$,
$i=1,\dots,n+1$, so that $\mathcal{B}_i(z)$ has $n$ elements and
$\mathcal{B}_{n+1}(z)=\aleph_n$. The associated scores are
$R_i(z):=A(\mathcal{B}_i(z),X_i)$, $i=1,\dots,n+1$, and the conformal
$p$-value of $z$ is
\begin{equation}\label{eq:pvalue}
p_z:=\frac{|\mathcal{C}(z)|}{n+1},\qquad
\mathcal{C}(z):=\big\{i\in\{1,\dots,n+1\}:R_i(z)\ge R_{n+1}(z)\big\}.
\end{equation}
Note that $n+1\in\mathcal{C}(z)$ always. For $\alpha\in(0,1)$, the
conformal prediction set at level $1-\alpha$ is
$\gamma^\alpha(\aleph_n):=\{z\in\mathcal{X}:p_z>\alpha\}$.\footnote{For the
score used in this paper, $z\mapsto R_i(z)$ is continuous for each $i$
(each $R_i(z)$ is an order statistic of finitely many continuous functions
of $z$), so $z\mapsto p_z$ is upper semicontinuous and
$\gamma^\alpha(\aleph_n)$ is a Borel (indeed closed) set; the suprema over
$z$ appearing in Section~\ref{consistency} are then handled by a routine
approximation argument. We omit further measurability considerations.}
The following proposition is proved in \cite{vovk03}.

\begin{prop}\label{prop:conformal-validity}
Let $(X_1,\dots,X_{n+1})$ be exchangeable and let $A$ be a symmetric
non-conformity score. Then $P(X_{n+1}\in\gamma^\alpha(\aleph_n))\ge1-\alpha$.
\end{prop}

\begin{rem}\label{rem:no-upper-bound}
For non-conformity scores whose values are almost surely distinct, the
standard argument also yields the upper bound
$P(X_{n+1}\in\gamma^\alpha(\aleph_n))\le1-\alpha+\frac{1}{n+1}$, showing
that full conformal regions are not overly conservative. For the half-mass
radius score introduced below, however, ties $R_i=R_j$ occur with
\emph{positive} probability even when $P$ is absolutely continuous: if
$X_i$ and $X_j$ realize each other's $\mathbf{k}$-th nearest-neighbour
distance, then $R_i=R_j=\|X_i-X_j\|$. For this reason we only claim the
lower bound of Proposition~\ref{prop:conformal-validity}; the numerical
experiments of Section~\ref{experiments} indicate that the resulting
regions are nevertheless not conservative in practice.
\end{rem}

\section{The half-mass radius score}\label{robustconf}

Although our asymptotic theory in Section~\ref{consistency} is developed in
$\R^d$, the score and its basic properties only require a metric
structure, and the contamination analysis of Section~\ref{contamination} is
carried out at that level of generality. We therefore state the basic
definitions on a general metric space.

\begin{definition}\label{def:half-mass}
Let $(\mathcal{X},d)$ be a metric space, let $P$ be a Borel probability
measure on $\mathcal{X}$ and let $m\in[0,1)$. The \emph{$m$-mass radius} of
$P$ at $x\in\mathcal{X}$ is
$\delta_{P,m}(x):=\inf\{r>0:P(B(x,r))>m\}$, where $B(x,r)$ denotes the
closed ball of centre $x$ and radius $r$. The \emph{half-mass radius} is
$\dP:=\delta_{P,1/2}$.
\end{definition}

Since $\bigcup_{r>0}B(x,r)=\mathcal{X}$, continuity from below gives
$P(B(x,r))\to1$ as $r\to\infty$, so $\delta_{P,m}(x)<\infty$ for every $x$
and every $m<1$. Moreover $\delta_{P,m}$ is $1$-Lipschitz: if $d(x,y)\le t$
then $B(y,r)\subset B(x,r+t)$ for every $r$, whence
$\delta_{P,m}(x)\le\delta_{P,m}(y)+t$, and the claim follows by symmetry.
In particular $\dP$ is continuous. The functional $\dP$ is closely related
to the distance-to-a-measure of \cite{chazal2011geometric}, with the mass
parameter fixed at the median level $m=1/2$; this choice is what produces
the median-type robustness quantified in Section~\ref{contamination}.

Given a bag $\mathcal{B}=(y_1,\dots,y_n)$ of points of $\mathcal{X}$ and
$z\in\mathcal{X}$, define the non-conformity score
\begin{equation}\label{funcA}
A(\mathcal{B},z):=\min_{\substack{I\subset\{1,\dots,n\}\\ |I|>n/2}}\;
\max_{j\in I}\, d(z,y_j).
\end{equation}
Since enlarging $I$ can only increase the inner maximum, the minimum in
\eqref{funcA} is attained at a set of cardinality exactly
$\mathbf{k}=\lfloor n/2\rfloor+1$, the smallest integer exceeding $n/2$.
For $z\in\mathcal{X}$, let $d_{(1)}(z)\le d_{(2)}(z)\le\dots\le d_{(n)}(z)$
denote the distances $d(z,y_j)$, $j=1,\dots,n$, sorted in nondecreasing
order and counted with multiplicities.

\begin{lem}\label{prop:minmax=knn}
For every bag $\mathcal{B}=(y_1,\dots,y_n)$ and every $z\in\mathcal{X}$,
$A(\mathcal{B},z)=d_{(\mathbf{k})}(z)$, that is, $A(\mathcal{B},z)$ is the
distance from $z$ to its $\mathbf{k}$-th nearest neighbour in
$\mathcal{B}$, with multiplicities.
\end{lem}

\begin{proof}
Let $I_0$ collect the indices of $\mathbf{k}$ smallest distances (ties
broken arbitrarily). Then $|I_0|=\mathbf{k}>n/2$ and the maximum of
$d(z,y_j)$ over $j\in I_0$ equals $d_{(\mathbf{k})}(z)$, so
$A(\mathcal{B},z)\le d_{(\mathbf{k})}(z)$. Conversely, any $I$ with
$|I|>n/2$ has $|I|\ge\mathbf{k}$, and at most $\mathbf{k}-1$ indices $j$
satisfy $d(z,y_j)<d_{(\mathbf{k})}(z)$; hence $I$ contains some $j$ with
$d(z,y_j)\ge d_{(\mathbf{k})}(z)$ and
$\max_{j\in I}d(z,y_j)\ge d_{(\mathbf{k})}(z)$. Taking the minimum over $I$
gives $A(\mathcal{B},z)\ge d_{(\mathbf{k})}(z)$.
\end{proof}

\begin{cor}\label{cor:A-empirical-radius}
Let $P^{\mathcal{B}}:=\frac1n\sum_{j=1}^n\delta_{y_j}$ be the empirical
measure of the bag $\mathcal{B}$. Then, for every $z\in\mathcal{X}$,
$A(\mathcal{B},z)=\delta_{P^{\mathcal{B}}}(z)$, the empirical half-mass
radius of $z$.
\end{cor}

\begin{proof}
If $r<d_{(\mathbf{k})}(z)$ then at most $\mathbf{k}-1$ points of
$\mathcal{B}$ lie in $B(z,r)$, so
$P^{\mathcal{B}}(B(z,r))\le(\mathbf{k}-1)/n\le1/2$ and $r$ is not feasible
in Definition~\ref{def:half-mass}. If $r\ge d_{(\mathbf{k})}(z)$ then at
least $\mathbf{k}$ points lie in $B(z,r)$ and
$P^{\mathcal{B}}(B(z,r))\ge\mathbf{k}/n>1/2$. Hence
$\delta_{P^{\mathcal{B}}}(z)=d_{(\mathbf{k})}(z)$, and the claim follows
from Lemma~\ref{prop:minmax=knn}.
\end{proof}

In words: the conformal machinery of Section~\ref{confinf}, run with the
score \eqref{funcA}, scores each point by the radius of the smallest ball
around it containing a strict majority of the remaining sample.
Figure~\ref{fig:halfmass} illustrates the score.

\begin{figure}[ht]
\centering
\includegraphics[width=0.5\textwidth]{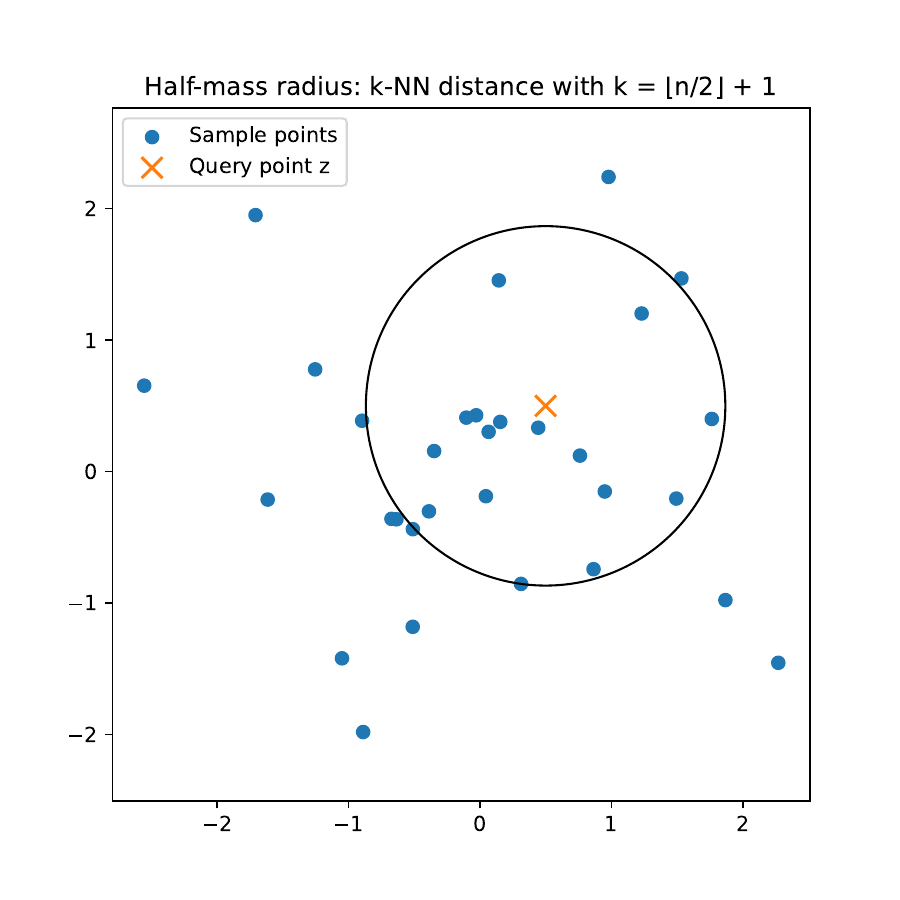}
\caption{The half-mass radius at a point $z$: the smallest radius $r$ such
that the closed ball $B(z,r)$ captures a strict majority of the sample.}
\label{fig:halfmass}
\end{figure}

\section{The population target}\label{population-target}

\begin{definition}[Central regions and calibrated level]\label{def:central}
Let $P$ be a probability measure on $\R^d$ and let $X\sim P$. For
$\beta\ge0$, let $Q_\beta:=\{x\in\R^d:\dP(x)\le\beta\}$; the family
$(Q_\beta)_{\beta\ge0}$ is increasing. For $\alpha\in(0,1)$, the
\emph{calibrated level} is
$\beta_\alpha:=\inf\{\beta\ge0:P(Q_\beta)\ge1-\alpha\}$, and $\Qba$ is
called the \emph{calibrated central region}. Since
$P(Q_\beta)=P(\dP(X)\le\beta)$, the map $\beta\mapsto P(Q_\beta)$ is the
distribution function of $\dP(X)$ and $\beta_\alpha$ is its
$(1-\alpha)$-quantile. By right-continuity, $P(\Qba)\ge1-\alpha$; if
$P(\dP(X)=\beta_\alpha)=0$, then $P(\Qba)=1-\alpha$.
\end{definition}

The sets $Q_\beta$ are compact (Lemma~\ref{lem:Qbeta-compact}). For
$0<u<1/2$ and $y\in\R^d$, define the \emph{mass-radius interval}
$I_u(y):=[\delta_{P,1/2-u}(y),\delta_{P,1/2+u}(y)]$. It contains $\dP(y)$,
which corresponds to the median mass level $m=1/2$, and it measures how
much the radius changes when the target mass is perturbed from $1/2-u$ to
$1/2+u$. The first assumption requires this sensitivity to remain
uniformly controlled near the calibrated level $\beta_\alpha$: points whose
mass-radius interval reaches a neighbourhood of $\beta_\alpha$ should not
experience large jumps when the mass level is slightly perturbed.

\begin{assumption}[Local stability of the half-mass radius]\label{H1}
Let $\alpha\in(0,1)$. There exist constants $\kappa>0$, $\rho_0>0$ and
$u_0\in(0,1/2)$ such that, for every $u\in(0,u_0]$ and every $y\in\R^d$
with $I_u(y)\cap[\beta_\alpha-\rho_0,\beta_\alpha+\rho_0]\ne\emptyset$, we
have $\delta_{P,1/2+u}(y)-\delta_{P,1/2-u}(y)\le\kappa u$.
\end{assumption}

Since $\dP(y)\in I_u(y)$, the condition covers in particular every $y$ with
$|\dP(y)-\beta_\alpha|\le\rho_0$; it is, however, strictly stronger, as it
also controls points whose half-mass radius lies outside the band but whose
mass-radius interval reaches it. This stronger form is what the proofs of
Theorems~\ref{thm:consistency-corrected-omega}
and~\ref{thm:outer-one-sided} and Proposition~\ref{prop:Qhat-H-conv}
actually use.

Our second assumption requires the distribution of $\dP(X)$ to be regular
around the quantile $\beta_\alpha$. Informally, there should be neither a
gap nor an excessive concentration of probability mass near
$\beta_\alpha$.

\begin{assumption}\label{H2}
Let $X\sim P$. There exist constants $L,c_-,c_+,t_0>0$ such that, for every
$t\in(0,t_0]$,
\[
\begin{aligned}
P(\beta_\alpha-t<\dP(X)\le\beta_\alpha)&\ge c_-t,\\
P(\beta_\alpha<\dP(X)\le\beta_\alpha+t)&\ge c_+t,\\
P(\beta_\alpha-t<\dP(X)\le\beta_\alpha+t)&\le Lt.
\end{aligned}
\]
(so that necessarily $c_-+c_+\le L$).
\end{assumption}

The upper bound implies $P(\dP(X)=\beta_\alpha)=0$, so that
$P(\dP(X)>\beta_\alpha)=\alpha$. The lower bounds ensure that the
distribution of $\dP(X)$ places positive mass on both sides of the
quantile $\beta_\alpha$.

\begin{assumption}\label{H3}
For every $\varepsilon>0$ there exists $\lambda>0$ such that
$|\beta-\beta_\alpha|\le\lambda$ implies $d_H(Q_\beta,\Qba)\le\varepsilon$,
where $d_H$ denotes the Hausdorff distance, with the convention
$d_H(\emptyset,A):=+\infty$ for $A\ne\emptyset$.
\end{assumption}

This condition prevents the appearance of macroscopic jumps in the central
regions when the threshold crosses the level $\beta_\alpha$: small
perturbations of the calibrated radius should only produce small
perturbations of the corresponding central set. When $\beta_\alpha>0$,
Proposition~\ref{prop:H3-characterization} shows that \ref{H3} is
equivalent to the absence of flat regions of $\dP$ at height
$\beta_\alpha$, namely to $\Qba=\overline{\{x:\dP(x)<\beta_\alpha\}}$.
Thus \ref{H3} excludes situations in which an entire region appears
instantaneously when the threshold reaches $\beta_\alpha$.

\begin{rem}\label{rem:beta-positive}
If $P$ is atomless then $\dP>0$ everywhere: indeed
$P(B(x,r))\downarrow P(\{x\})=0$ as $r\downarrow0$, so small balls around
$x$ carry mass smaller than $1/2$. Since $\dP$ is continuous and coercive
($\dP(x)\ge d(x,K)\to\infty$, as in the proof of
Lemma~\ref{lem:Qbeta-compact}), it attains a strictly positive minimum, and
hence $\beta_\alpha\ge\min\dP>0$. Alternatively, \ref{H2} by itself forces
$\beta_\alpha>0$: if $\beta_\alpha=0$, the left lower bound would give
$P(\dP(X)=0)=P(-t<\dP(X)\le0)\ge c_-t>0$ for all small $t$, while the
two-sided upper bound gives $P(\dP(X)=0)\le P(-t<\dP(X)\le t)\le Lt\to0$,
a contradiction.
\end{rem}

\section{Robustness under contamination}\label{contamination}

In this section $(\mathcal{X},d)$ is a general metric space. Let $P_0$ be a
Borel probability on $\mathcal{X}$ (the \emph{clean law}) and, for
$0\le\eta<1/2$ and an arbitrary Borel probability $Q$ (the
\emph{contamination}), consider the mixture $P_\eta:=(1-\eta)P_0+\eta Q$.
Write $\delta_0:=\delta_{P_0,1/2}$ and $\delta_\eta:=\delta_{P_\eta,1/2}$
for the corresponding half-mass radii, and
$Q^0_t:=\{x:\delta_0(x)\le t\}$, $Q^\eta_t:=\{x:\delta_\eta(x)\le t\}$,
$t\in\R$. All the results of this section hold uniformly over the
contamination $Q$.

This contamination model is related to recent analyses of conformal
procedures with corrupted calibration or reference data
\citep{clarkson2024split,bashari2025robust}. Our perspective is different:
rather than studying the effect of contaminated calibration scores on the
coverage of a fixed conformal procedure, we analyse the population
stability of the half-mass radius itself, the induced calibrated level and
the corresponding central set, uniformly over arbitrary contaminations $Q$.

\subsection{Stability of the half-mass radius}

The following result compares the half-mass radii of $P_\eta$ and $P_0$
for every contamination level $\eta<1/2$.

\begin{prop}\label{prop:contamination-bracketing}
For every $x\in\mathcal{X}$ and every $0\le\eta<1/2$,
\[
\delta_{P_0,a_\eta}(x)\le\delta_\eta(x)\le\delta_{P_0,b_\eta}(x),
\qquad
a_\eta:=\frac{1-2\eta}{2(1-\eta)}=\frac12-u_\eta,\quad
b_\eta:=\frac{1}{2(1-\eta)}=\frac12+u_\eta,
\]
where $u_\eta:=\eta/(2(1-\eta))$. Note that $u_\eta\le\eta$ for
$0\le\eta<1/2$, and that $b_\eta<1$ precisely because $\eta<1/2$.
\end{prop}

\begin{proof}
If $P_\eta(B(x,r))>1/2$ then
$(1-\eta)P_0(B(x,r))\ge P_\eta(B(x,r))-\eta>1/2-\eta$, so
$P_0(B(x,r))>a_\eta$; hence every radius feasible for $\delta_\eta(x)$ is
feasible for $\delta_{P_0,a_\eta}(x)$, and
$\delta_{P_0,a_\eta}(x)\le\delta_\eta(x)$. Conversely, if
$P_0(B(x,r))>b_\eta$ then
$P_\eta(B(x,r))\ge(1-\eta)P_0(B(x,r))>(1-\eta)b_\eta=1/2$, so every radius
feasible for $\delta_{P_0,b_\eta}(x)$ is feasible for $\delta_\eta(x)$, and
$\delta_\eta(x)\le\delta_{P_0,b_\eta}(x)$.
\end{proof}

\begin{rem}\label{rem:bracketing-sharp}
For fixed $x$, both brackets in
Proposition~\ref{prop:contamination-bracketing} are attained. If
$Q=\delta_x$ (a point mass at $x$), then
$P_\eta(B(x,r))=(1-\eta)P_0(B(x,r))+\eta$, which exceeds $1/2$ exactly when
$P_0(B(x,r))>a_\eta$; hence $\delta_\eta(x)=\delta_{P_0,a_\eta}(x)$. If
instead $Q$ is concentrated at distance greater than
$\delta_{P_0,b_\eta}(x)$ from $x$, then
$P_\eta(B(x,r))=(1-\eta)P_0(B(x,r))$ for all radii
$r\le\delta_{P_0,b_\eta}(x)$, and $\delta_\eta(x)=\delta_{P_0,b_\eta}(x)$.
Thus the bracketing is sharp: as $Q$ varies, $\delta_\eta(x)$ attains both
endpoints of the interval
$[\delta_{P_0,1/2-u_\eta}(x),\delta_{P_0,1/2+u_\eta}(x)]$.
\end{rem}

Proposition~\ref{prop:contamination-bracketing}, which is sharp by
Remark~\ref{rem:bracketing-sharp}, converts the contamination problem
into a question about the regularity of the map $m\mapsto\delta_{P_0,m}(x)$
near $m=1/2$. We quantify this through the following local condition on
the clean law at a level $\beta^*>0$: there exist $\kappa_0>0$, $\rho_0>0$
and $u_0\in(0,1/2)$ such that, for every $u\in(0,u_0]$ and every
$x\in\mathcal{X}$ whose interval
$I^0_u(x):=[\delta_{P_0,1/2-u}(x),\delta_{P_0,1/2+u}(x)]$ intersects
$[\beta^*-\rho_0,\beta^*+\rho_0]$,
\begin{equation}\label{eq:local-strip-condition-contamination}
\delta_{P_0,1/2+u}(x)-\delta_{P_0,1/2-u}(x)\le\kappa_0\,u .
\end{equation}
When $\beta^*=\beta^0_\alpha$ is the calibrated level of the clean law
(defined below), \eqref{eq:local-strip-condition-contamination} is exactly
Assumption~\ref{H1} for $P_0$. Throughout, we write
$\rho_\eta:=\kappa_0u_\eta$ ($\le\kappa_0\eta$). Under
\eqref{eq:local-strip-condition-contamination},
Lemma~\ref{cor:contamination-local-stability} in Appendix~\ref{app} shows
that $|\delta_\eta-\delta_0|\le\rho_\eta$ on the relevant level strip and
that $Q^0_{\beta-\rho_\eta}\subset Q^\eta_\beta\subset
Q^0_{\beta+\rho_\eta}$ for levels $\beta$ close to $\beta^*$.

\subsection{Stability of the calibrated level and of the calibrated set}

For a Borel probability $P'$ on $\mathcal{X}$ define
$\beta_\alpha(P'):=\inf\{\beta\in\R:P'(\{x:\delta_{P',1/2}(x)\le\beta\})\ge1-\alpha\}$,
and write $\beta^0_\alpha:=\beta_\alpha(P_0)$,
$\beta^\eta_\alpha:=\beta_\alpha(P_\eta)$. Since the half-mass radius is
nonnegative and $1-\alpha>0$, one always has $\beta_\alpha(P')\ge0$.

\begin{lem}\label{lem:quantile-stability}
Assume that (a) $P_0$ satisfies
\eqref{eq:local-strip-condition-contamination} at
$\beta^*=\beta^0_\alpha$ with constants $\kappa_0,\rho_0>0$ and
$u_0\in(0,1/2)$, and (b) $P_0$ satisfies Assumption~\ref{H2} at
$\beta^0_\alpha$ with constants $L,c_-,c_+,t_0>0$. Let
$c:=\min\{c_-,c_+\}$ and
\[
\eta_1:=\min\Big\{u_0,\ \frac{\min\{\rho_0,t_0,\beta^0_\alpha\}}{2\kappa_0+1/c}\Big\}.
\]
Note that (b) forces $\beta^0_\alpha>0$ (Remark~\ref{rem:beta-positive}),
so $\eta_1>0$, and that $u_\eta\le\eta$, so $\eta\le\eta_1$ implies
$u_\eta\le u_0$. Then, for every $0\le\eta\le\eta_1$ and every
contamination $Q$, we have $\beta^\eta_\alpha>0$ and
\[
|\beta^\eta_\alpha-\beta^0_\alpha|\le\rho_\eta+\frac{\eta}{c}
\le\Big(\kappa_0+\frac1c\Big)\eta .
\]
\end{lem}

\begin{proof}
The case $\eta=0$ is trivial; fix $0<\eta\le\eta_1$. Let $F_0$ denote the
distribution function of $\delta_0(X_0)$, $X_0\sim P_0$. The upper bound
in (b) gives $P(\delta_0(X_0)=\beta^0_\alpha)=0$; combined with the
right-continuity of $F_0$ and with $F_0(t)<1-\alpha$ for
$t<\beta^0_\alpha$ (definition of the infimum), this yields
$F_0(\beta^0_\alpha)=1-\alpha$. Throughout we use
$P_0(A)-\eta\le P_\eta(A)\le P_0(A)+\eta$, valid for every Borel set $A$
since $P_\eta=(1-\eta)P_0+\eta Q$.

Set $t_+:=\rho_\eta+\eta/c_+$ and $\beta:=\beta^0_\alpha+t_+$. Since
$\rho_\eta\le\kappa_0\eta$ and $\eta/c_+\le\eta/c$, we have
$t_+\le(\kappa_0+1/c)\eta\le(2\kappa_0+1/c)\eta\le\min\{\rho_0,t_0,\beta^0_\alpha\}$,
so $|\beta-\beta^0_\alpha|\le\rho_0$ and $\beta>0$. By
Lemma~\ref{cor:contamination-local-stability}(ii),
$Q^\eta_\beta\supset Q^0_{\beta-\rho_\eta}=Q^0_{\beta^0_\alpha+\eta/c_+}$,
and therefore
\[
P_\eta(Q^\eta_\beta)\ge P_0\big(Q^0_{\beta^0_\alpha+\eta/c_+}\big)-\eta
=F_0\Big(\beta^0_\alpha+\frac{\eta}{c_+}\Big)-\eta
\ge(1-\alpha)+c_+\cdot\frac{\eta}{c_+}-\eta=1-\alpha ,
\]
where the right-mass lower bound in (b) applies because
$0<\eta/c_+\le(2\kappa_0+1/c)\eta\le t_0$. Hence
$\beta^\eta_\alpha\le\beta=\beta^0_\alpha+\rho_\eta+\eta/c_+$.

Fix $0<\zeta\le\kappa_0\eta$, set $s:=\eta/c_-+\zeta$ and
$\beta':=\beta^0_\alpha-\rho_\eta-s$. If $\beta'\le0$ there is nothing to
prove, since $\beta^\eta_\alpha\ge0\ge\beta'$. Otherwise,
$\rho_\eta+s\le\kappa_0\eta+\eta/c+\kappa_0\eta=(2\kappa_0+1/c)\eta\le\min\{\rho_0,t_0\}$,
so $|\beta'-\beta^0_\alpha|\le\rho_0$, $s\le t_0$, and
Lemma~\ref{cor:contamination-local-stability}(ii) applies at $\beta'>0$,
giving $Q^\eta_{\beta'}\subset Q^0_{\beta'+\rho_\eta}=Q^0_{\beta^0_\alpha-s}$.
Therefore, by the left-mass lower bound in (b), applied with $t=s$, and
$F_0(\beta^0_\alpha)=1-\alpha$,
\[
P_\eta(Q^\eta_{\beta'})\le F_0(\beta^0_\alpha-s)+\eta
\le F_0(\beta^0_\alpha)-c_-s+\eta=1-\alpha-c_-\zeta<1-\alpha .
\]
By monotonicity of $\beta\mapsto P_\eta(Q^\eta_\beta)$ and the definition
of the infimum, $\beta^\eta_\alpha\ge\beta'$. Letting $\zeta\downarrow0$
gives $\beta^\eta_\alpha\ge\beta^0_\alpha-\rho_\eta-\eta/c_-$.

Combining the two bounds,
$|\beta^\eta_\alpha-\beta^0_\alpha|\le\rho_\eta+\eta/c$. Finally, since
$u_\eta<\eta$ for $0<\eta<1/2$, we have
$\rho_\eta+\eta/c=\kappa_0u_\eta+\eta/c<(2\kappa_0+1/c)\eta\le\beta^0_\alpha$,
whence $\beta^\eta_\alpha\ge\beta^0_\alpha-(\rho_\eta+\eta/c)>0$.
\end{proof}

\begin{thm}\label{cor:calibrated-set-stability}
Let $X_0\sim P_0$ and assume (a) and (b) of
Lemma~\ref{lem:quantile-stability}. Then, for every $0\le\eta\le\eta_1$
and every contamination $Q$,
\[
P\big(X_0\in Q^\eta_{\beta^\eta_\alpha}\,\triangle\,Q^0_{\beta^0_\alpha}\big)
\le L\Big(2\kappa_0+\frac1c\Big)\eta .
\]
\end{thm}

\begin{proof}
The case $\eta=0$ is trivial; fix $0<\eta\le\eta_1$ and set
$\Delta:=|\beta^\eta_\alpha-\beta^0_\alpha|$. By
Lemma~\ref{lem:quantile-stability},
$\Delta\le\rho_\eta+\eta/c\le(\kappa_0+1/c)\eta\le\rho_0$ and
$\beta^\eta_\alpha>0$; together with $u_\eta\le u_0$, this shows that
Lemma~\ref{cor:contamination-local-stability}(ii) applies at
$\beta=\beta^\eta_\alpha$, giving
$Q^0_{\beta^\eta_\alpha-\rho_\eta}\subset Q^\eta_{\beta^\eta_\alpha}\subset
Q^0_{\beta^\eta_\alpha+\rho_\eta}$. Set
$r:=\Delta+\rho_\eta\le(\kappa_0+1/c)\eta+\kappa_0u_\eta\le(2\kappa_0+1/c)\eta\le t_0$.
If $x\in Q^\eta_{\beta^\eta_\alpha}\setminus Q^0_{\beta^0_\alpha}$ then
$\delta_0(x)>\beta^0_\alpha$ and
$\delta_0(x)\le\beta^\eta_\alpha+\rho_\eta\le\beta^0_\alpha+r$; if
$x\in Q^0_{\beta^0_\alpha}\setminus Q^\eta_{\beta^\eta_\alpha}$ then
$\delta_0(x)\le\beta^0_\alpha$ and
$\delta_0(x)>\beta^\eta_\alpha-\rho_\eta\ge\beta^0_\alpha-r$. Hence
$Q^\eta_{\beta^\eta_\alpha}\triangle Q^0_{\beta^0_\alpha}\subset\{x:
\beta^0_\alpha-r<\delta_0(x)\le\beta^0_\alpha+r\}$, whose $P_0$-mass is at
most $Lr\le L(2\kappa_0+1/c)\eta$ by the upper bound in (b).
\end{proof}

Lemma~\ref{cor:contamination-set-stability} (in Appendix~\ref{app}) and
Theorem~\ref{cor:calibrated-set-stability} correspond to two distinct
operational scenarios: the former fixes the threshold on clean data and
measures how the \emph{set} moves when the underlying law is contaminated,
whereas the latter recalibrates the threshold on the contaminated law
itself, which is what happens when the whole pipeline (including
calibration) is run on contaminated data. Both degrade linearly in $\eta$,
with explicit constants depending only on the clean law.

\section{Consistency, exponential bounds and rates}\label{consistency}

From now on $\mathcal{X}=\R^d$ with the Euclidean distance,
$X,X_1,X_2,\dots$ are i.i.d.\ with law $P$, and $X$ is independent of
$\aleph_n=(X_1,\dots,X_n)$. We write $N:=n+1$.

\subsection{Empirical process tools}

Let $\mathcal{A}_d$ denote the class of closed Euclidean balls of $\R^d$.
Its Vapnik--Chervonenkis dimension is $d+1$ \citep{dudley79}, so by
Sauer's lemma its shatter coefficients satisfy
$s(\mathcal{A}_d,m)\le m^{d+1}+1$ for every $m\ge1$.

\begin{thm}\label{thm:thluc}
There is a universal constant $C>0$ such that, for every $n\ge1$ and every
$0<\varepsilon\le1$,
\[
P\Big(\sup_{A\in\mathcal{A}_d}|P^{\aleph_n}(A)-P(A)|>\varepsilon\Big)
\le C\big(n^{2d+2}+1\big)\,e^{-2n\varepsilon^2}.
\]
\end{thm}

\begin{proof}
This is Devroye's refinement of the Vapnik--Chervonenkis inequality, see
\cite[Theorem~12.8]{pattern}, whose bound involves the shatter coefficient
$s(\mathcal{A}_d,n^2)\le n^{2(d+1)}+1$ and a factor
$4e^{4\varepsilon+4\varepsilon^2}\le4e^8$ for $\varepsilon\le1$, which we
absorb into $C$.\footnote{If the inequality of
\cite[Theorem~12.8]{pattern} is invoked under a side condition of the form
$n\varepsilon^2\ge c_0$ for an absolute constant $c_0$, the displayed bound
remains valid for \emph{all} $n\ge1$ and $0<\varepsilon\le1$ with
$C=\max\{4e^8,e^{2c_0}\}$: indeed, when $n\varepsilon^2<c_0$ the right-hand
side exceeds $e^{2c_0}e^{-2n\varepsilon^2}>1$ and the inequality holds
trivially.}
\end{proof}

The next deterministic lemma converts uniform deviations over balls into a
two-sided bracketing of mass radii.

\begin{lem}\label{lem:deterministic-bracketing}
Let $P,\tilde P$ be Borel probabilities on $\R^d$, let $m\in(0,1)$ and
$0<\theta<\min\{m,1-m\}$. If
$\sup_{A\in\mathcal{A}_d}|\tilde P(A)-P(A)|\le\theta$, then
$\delta_{P,m-\theta}(y)\le\delta_{\tilde P,m}(y)\le\delta_{P,m+\theta}(y)$
for every $y\in\R^d$.
\end{lem}

\begin{proof}
For the upper bound, let $s>\delta_{P,m+\theta}(y)$; by definition of the
infimum there is $r\le s$ with $P(B(y,r))>m+\theta$, whence
$\tilde P(B(y,r))\ge P(B(y,r))-\theta>m$ and
$\delta_{\tilde P,m}(y)\le r\le s$. Taking the infimum over such $s$ gives
the upper bound. For the lower bound, every $r<\delta_{P,m-\theta}(y)$
satisfies $P(B(y,r))\le m-\theta$, hence $\tilde P(B(y,r))\le m$, so no
such $r$ is feasible for $\delta_{\tilde P,m}(y)$.
\end{proof}

Recall, for $z\in\R^d$, the leave-one-out bags $\mathcal{B}_i(z)$,
$i=1,\dots,n+1$, of Section~\ref{confinf}, each of cardinality $n$, with
$\mathcal{B}_{n+1}(z)=\aleph_n$. For $\theta>0$ define the event
\[
R^*_n(\theta):=\Big\{\sup_{z\in\R^d}\ \max_{1\le i\le n+1}\
\sup_{A\in\mathcal{A}_d}|P^{\mathcal{B}_i(z)}(A)-P(A)|\le\theta\Big\}.
\]

\begin{prop}\label{prop:uniform_hp}
For every $\theta$ with $1/n<\theta<1$,
\[
P\big(R^*_n(\theta)^c\big)\le C\big(n^{2d+2}+1\big)\,e^{-2n(\theta-1/n)^2}.
\]
If moreover $n\theta\ge2$, then $e^{-2n(\theta-1/n)^2}\le e^{-n\theta^2/2}$.
\end{prop}

\begin{proof}
Each bag $\mathcal{B}_i(z)$ has $n$ elements and differs from $\aleph_n$ in
at most one of them ($X_i$ is replaced by $z$ when $i\le n$;
$\mathcal{B}_{n+1}(z)=\aleph_n$). Hence, for every $A$, $i$ and $z$,
$|P^{\mathcal{B}_i(z)}(A)-P^{\aleph_n}(A)|\le1/n$, and therefore
$\sup_{z,i,A}|P^{\mathcal{B}_i(z)}(A)-P(A)|\le\Delta_n+1/n$, where
$\Delta_n:=\sup_{A\in\mathcal{A}_d}|P^{\aleph_n}(A)-P(A)|$. Thus
$R^*_n(\theta)^c\subset\{\Delta_n>\theta-1/n\}$ and the first claim follows
from Theorem~\ref{thm:thluc}. If $n\theta\ge2$ then
$\theta-1/n\ge\theta/2$, whence $2n(\theta-1/n)^2\ge n\theta^2/2$.
\end{proof}

\subsection{Consistency of the conformal region}

Denote by $F$ the distribution function of $\dP(X)$ and define, for
$\varepsilon>0$,
$\omega(\varepsilon):=P(\beta_\alpha-\varepsilon<\dP(X)\le\beta_\alpha+\varepsilon)$
and $p^\mp_\varepsilon:=P(\dP(X)>\beta_\alpha\mp\varepsilon/2)$. By
definition of $\beta_\alpha$ as an infimum,
$F(\beta_\alpha-\varepsilon/2)<1-\alpha$, so $p^-_\varepsilon>\alpha$; and
by monotonicity and right-continuity,
$F(\beta_\alpha+\varepsilon/2)\ge F(\beta_\alpha)\ge1-\alpha$, so
$p^+_\varepsilon\le\alpha$.

\begin{thm}\label{thm:consistency-corrected-omega}
Assume \ref{H1} and let $\varepsilon$ and $n$ satisfy
$0<\varepsilon\le\rho_0$ and $1/n<\varepsilon/(4\kappa)\le u_0$. Then
\begin{multline*}
P\big(X\in\gamma^\alpha(\aleph_n)\,\triangle\,\Qba\big)
\le\omega(\varepsilon)
+\exp\Big\{-\frac{N(p^-_\varepsilon-\alpha)^2}{2p^-_\varepsilon}\Big\}
+\exp\big\{-2N(\alpha-p^+_\varepsilon)^2\big\}\\
+2C\big(n^{2d+2}+1\big)
\exp\Big\{-2n\Big(\frac{\varepsilon}{4\kappa}-\frac1n\Big)^2\Big\}.
\end{multline*}
\end{thm}

Under Assumption~\ref{H2}, the gaps $p^-_\varepsilon-\alpha$ and
$\alpha-p^+_\varepsilon$ are bounded below by multiples of $\varepsilon$
(Remark~\ref{rem:pplus-trivial}), and
Corollary~\ref{cor:rate-symmetric-difference} below turns this bound into
an explicit rate.

\begin{proof}
Set $\theta:=\varepsilon/(4\kappa)$, so that $1/n<\theta\le u_0<1/2$ by
hypothesis, write $\delta^{\mp}(y):=\delta_{P,1/2\mp\theta}(y)$, and let
$G:=R_n^*(\theta)$, so that
$P(G^c)\le C(n^{2d+2}+1)\exp\{-2n(\varepsilon/(4\kappa)-1/n)^2\}$ by
Proposition~\ref{prop:uniform_hp}. On $G$, every leave-one-out empirical
measure $P^{\mathcal{B}_i(X)}$, $i=1,\dots,N$, is uniformly $\theta$-close
to $P$ on balls, so Lemma~\ref{lem:deterministic-bracketing} (applied with
$m=1/2$) traps every empirical half-mass radius between two deterministic
ones:
\begin{equation}\label{eq:trap}
\delta^{-}(y)\le\delta_{P^{\mathcal{B}_i(X)}}(y)\le\delta^{+}(y)
\qquad\text{for every } y\in\R^d \text{ and every } i\le N.
\end{equation}
In particular, by Corollary~\ref{cor:A-empirical-radius}, on $G$ each score
satisfies $R_i(X)\in[\delta^{-}(X_i),\delta^{+}(X_i)]$, where we write
$X_N:=X$.

The role of \ref{H1} is to make these trapping intervals short near the
calibrated level: for every $\lambda$ with $|\lambda-\beta_\alpha|\le\rho_0$
and every $y\in\R^d$,
\begin{equation}\label{eq:margin}
\text{(a)}\ \ \dP(y)\le\lambda-\tfrac{\varepsilon}{4}
\ \Longrightarrow\ \delta^{+}(y)\le\lambda,
\qquad
\text{(b)}\ \ \dP(y)>\lambda+\tfrac{\varepsilon}{4}
\ \Longrightarrow\ \delta^{-}(y)>\lambda.
\end{equation}
Indeed, if the conclusion of (a) failed we would have
$\delta^{-}(y)\le\dP(y)\le\lambda-\varepsilon/4<\lambda<\delta^{+}(y)$, so
$\lambda$ would be caught inside $I_\theta(y)$; the interval would then meet
$[\beta_\alpha-\rho_0,\beta_\alpha+\rho_0]$, \ref{H1} (with
$u=\theta\le u_0$) would bound its length by $\kappa\theta=\varepsilon/4$,
and hence $\delta^{+}(y)\le\delta^{-}(y)+\varepsilon/4\le\lambda$, a
contradiction. The same argument proves (b): if $\delta^{-}(y)\le\lambda$,
then $\lambda\in I_\theta(y)$, its length is at most $\varepsilon/4$, and
$\dP(y)\le\delta^{-}(y)+\varepsilon/4\le\lambda+\varepsilon/4$, again a
contradiction. We use \eqref{eq:margin} at the two levels
$\lambda_\mp:=\beta_\alpha\mp3\varepsilon/4$, both admissible since
$3\varepsilon/4<\varepsilon\le\rho_0$.

Let $E_1:=\{\dP(X)\le\beta_\alpha-\varepsilon,\ X\notin\gamma^\alpha(\aleph_n)\}$
and $E_2:=\{\dP(X)>\beta_\alpha+\varepsilon,\ X\in\gamma^\alpha(\aleph_n)\}$.
Since $X\in\Qba$ when $\dP(X)\le\beta_\alpha-\varepsilon$ and $X\notin\Qba$
when $\dP(X)>\beta_\alpha+\varepsilon$,
$\{X\in\gamma^\alpha(\aleph_n)\triangle\Qba\}\subset
\{\beta_\alpha-\varepsilon<\dP(X)\le\beta_\alpha+\varepsilon\}\cup E_1\cup E_2$.
The first event has probability $\omega(\varepsilon)$; we intersect $E_1$
and $E_2$ with $G$, at the total price of $2P(G^c)$, and it remains to
bound $P(G\cap E_1)$ and $P(G\cap E_2)$.

Suppose first that $G\cap E_1$ holds. Since
$\dP(X)\le\lambda_--\varepsilon/4$, implication (a) at $\lambda_-$ together
with \eqref{eq:trap} gives $R_N(X)\le\delta^{+}(X)\le\lambda_-$, while every
$i\le n$ with $\dP(X_i)>\beta_\alpha-\varepsilon/2=\lambda_-+\varepsilon/4$
satisfies $R_i(X)\ge\delta^{-}(X_i)>\lambda_-\ge R_N(X)$ by (b) at
$\lambda_-$, and hence belongs to $\mathcal{C}(X)$. As
$N\in\mathcal{C}(X)$ always,
$|\mathcal{C}(X)|\ge1+\#\{i\le n:\dP(X_i)>\beta_\alpha-\varepsilon/2\}$.
Let $Y_N:=\#\{i\le N:\dP(X_i)>\beta_\alpha-\varepsilon/2\}$; counting $X$
itself makes the number of trials exactly $N$, so that
$Y_N\sim\mathrm{Bin}(N,p^-_\varepsilon)$, and on the present event the
indicator of $X$ vanishes because
$\dP(X)\le\beta_\alpha-\varepsilon<\beta_\alpha-\varepsilon/2$; therefore
$Y_N\le|\mathcal{C}(X)|-1$. Since $X\notin\gamma^\alpha(\aleph_n)$, that is
$|\mathcal{C}(X)|\le\alpha N$, we get $Y_N\le\alpha N-1<\alpha N$. As
$p^-_\varepsilon>\alpha$, the multiplicative Chernoff bound for the lower
tail, $P(\mathrm{Bin}(N,p)\le aN)\le\exp\{-N(p-a)^2/(2p)\}$ for $0<a<p$,
yields $P(G\cap E_1)\le\exp\{-N(p^-_\varepsilon-\alpha)^2/(2p^-_\varepsilon)\}$.

Suppose now that $G\cap E_2$ holds. Since
$\dP(X)>\lambda_++\varepsilon/4$, implication (b) at $\lambda_+$ and
\eqref{eq:trap} give $R_N(X)\ge\delta^{-}(X)>\lambda_+$. Every
$i\in\mathcal{C}(X)$ then satisfies
$\delta^{+}(X_i)\ge R_i(X)\ge R_N(X)>\lambda_+$, so the contrapositive of
(a) at $\lambda_+$ yields
$\dP(X_i)>\lambda_+-\varepsilon/4=\beta_\alpha+\varepsilon/2$. Hence
$\mathcal{C}(X)\subset\{i\le N:\dP(X_i)>\beta_\alpha+\varepsilon/2\}$,
whose cardinality $Y'_N$ is $\mathrm{Bin}(N,p^+_\varepsilon)$-distributed.
Since $X\in\gamma^\alpha(\aleph_n)$, we get
$\alpha N<|\mathcal{C}(X)|\le Y'_N$, and since $p^+_\varepsilon\le\alpha$,
Hoeffding's inequality gives
$P(G\cap E_2)\le\exp\{-2N(\alpha-p^+_\varepsilon)^2\}$. Summing the four
contributions completes the proof.
\end{proof}

\begin{rem}\label{rem:pplus-trivial}
The second exponential term in
Theorem~\ref{thm:consistency-corrected-omega} is informative only when
$p^+_\varepsilon<\alpha$. Under \ref{H2} this gap is quantitative: the
upper bound in \ref{H2} gives $P(\dP(X)=\beta_\alpha)=0$, whence (as in the
proof of Lemma~\ref{lem:quantile-stability})
$P(\dP(X)\le\beta_\alpha)=1-\alpha$ and $P(\dP(X)>\beta_\alpha)=\alpha$
exactly. Therefore, for $0<\varepsilon/2\le t_0$,
\[
\alpha-p^+_\varepsilon
=P\big(\beta_\alpha<\dP(X)\le\beta_\alpha+\tfrac{\varepsilon}{2}\big)
\ge\frac{c_+\varepsilon}{2},\qquad
p^-_\varepsilon-\alpha
=P\big(\beta_\alpha-\tfrac{\varepsilon}{2}<\dP(X)\le\beta_\alpha\big)
\ge\frac{c_-\varepsilon}{2},
\]
and $\omega(\varepsilon)\le L\varepsilon$ for $\varepsilon\le t_0$.
Figure~\ref{fig:gamma-conv} illustrates the convergence of the conformal
region to $\Qba$.
\end{rem}

\begin{figure}[ht]
\centering
\includegraphics[width=0.45\textwidth]{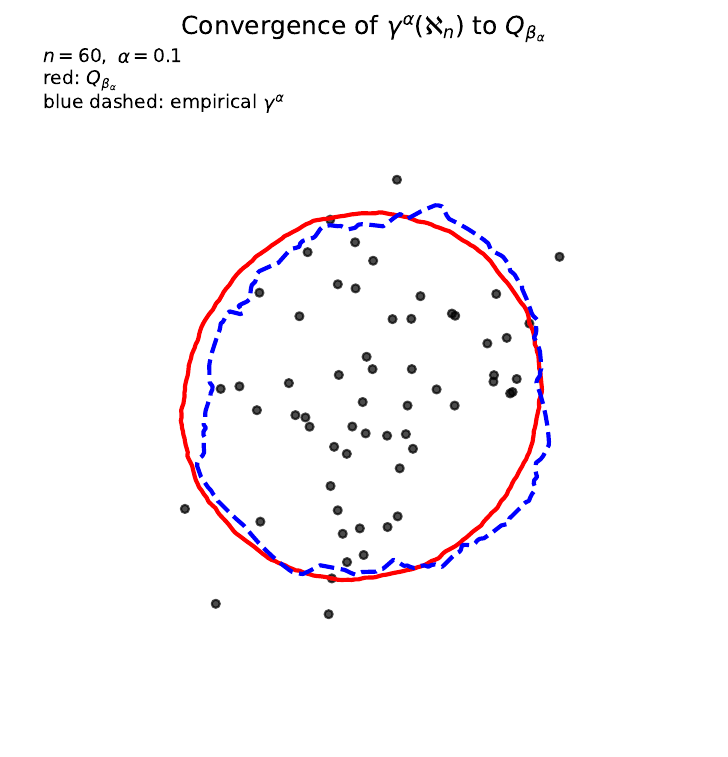}
\caption{Convergence of the conformal region $\gamma^\alpha(\aleph_n)$
(blue) to the population central set $\Qba$, whose boundary
$\{\dP=\beta_\alpha\}$ is shown in red, for increasing sample sizes.}
\label{fig:gamma-conv}
\end{figure}

\begin{cor}\label{cor:rate-symmetric-difference}
Assume \ref{H1} and \ref{H2}, let
$K\ge K_0:=\max\{2/c_-,\,1/c_+,\,\kappa\sqrt{32(2d+5/2)}\}$ and
$\varepsilon_n:=K(\log n/n)^{1/2}$. Then, for every $n$ large enough that
$\varepsilon_n\le\min\{\rho_0,t_0,4\kappa u_0\}$ and
$n\varepsilon_n\ge8\kappa$,
\[
P\big(X\in\gamma^\alpha(\aleph_n)\,\triangle\,\Qba\big)
\le L\varepsilon_n+2n^{-1/2}+4C\,n^{-1/2}
=O\big((\log n/n)^{1/2}\big).
\]
\end{cor}

\begin{proof}
Apply Theorem~\ref{thm:consistency-corrected-omega} with
$\varepsilon=\varepsilon_n$ (note that
$\theta_n:=\varepsilon_n/(4\kappa)\le u_0$ and $\theta_n\ge2/n>1/n$ by the
assumptions). By Remark~\ref{rem:pplus-trivial} and
$p^-_{\varepsilon_n}\le1$,
\begin{align*}
\exp\Big\{-\frac{N(p^-_{\varepsilon_n}-\alpha)^2}{2p^-_{\varepsilon_n}}\Big\}
&\le\exp\Big\{-\frac{(n+1)c_-^2\varepsilon_n^2}{8}\Big\}
\le n^{-K^2c_-^2/8}\le n^{-1/2},\\
\exp\big\{-2N(\alpha-p^+_{\varepsilon_n})^2\big\}
&\le\exp\Big\{-\frac{(n+1)c_+^2\varepsilon_n^2}{2}\Big\}
\le n^{-K^2c_+^2/2}\le n^{-1/2},
\end{align*}
where the exponents use $\varepsilon_n^2=K^2\log n/n$ and the choice of
$K_0$. For the empirical-process term, $n\theta_n\ge2$ gives
$2n(\theta_n-1/n)^2\ge n\theta_n^2/2=\frac{K^2}{32\kappa^2}\log n
\ge(2d+\frac52)\log n$, so
$2C(n^{2d+2}+1)e^{-2n(\theta_n-1/n)^2}\le4C\,n^{2d+2}n^{-(2d+5/2)}=4C\,n^{-1/2}$.
Finally $\omega(\varepsilon_n)\le L\varepsilon_n$ since
$\varepsilon_n\le t_0$, and $L\varepsilon_n$ dominates.
\end{proof}

\begin{thm}[One-sided outer bound at an inflated level]\label{thm:outer-one-sided}
Assume \ref{H1} and let $\varepsilon$ and $n$ satisfy
$0<\varepsilon<1-\alpha$, $\varepsilon\le2\rho_0$ and
$1/n<\varepsilon/(2\kappa)\le u_0$. Then
\begin{multline*}
P\big(X\in\gamma^{\alpha+\varepsilon}(\aleph_n)\setminus\Qba\big)
\le P\big(\beta_\alpha<\dP(X)\le\beta_\alpha+\varepsilon\big)
+e^{-2N\varepsilon^2}\\
+C\big(n^{2d+2}+1\big)
\exp\Big\{-2n\Big(\frac{\varepsilon}{2\kappa}-\frac1n\Big)^2\Big\}.
\end{multline*}
\end{thm}

\begin{proof}
Set $\theta:=\varepsilon/(2\kappa)$, so $\kappa\theta=\varepsilon/2$, write
$\delta^{\mp}(y):=\delta_{P,1/2\mp\theta}(y)$, and let $G:=R^*_n(\theta)$.
Arguing exactly as in the proof of
Theorem~\ref{thm:consistency-corrected-omega} (now with band margin
$\varepsilon/2\le\rho_0$ and width bound $\kappa\theta=\varepsilon/2$),
\ref{H1} yields: (I) if $\dP(y)>\beta_\alpha+\varepsilon$, then
$\delta^-(y)>\beta_\alpha+\varepsilon/2$; (II) if $\dP(y)\le\beta_\alpha$,
then $\delta^+(y)\le\beta_\alpha+\varepsilon/2$. Hence
$\{X\in\gamma^{\alpha+\varepsilon}(\aleph_n)\setminus\Qba\}\subset
\{\beta_\alpha<\dP(X)\le\beta_\alpha+\varepsilon\}\cup E\cup G^c$, where
$E:=\{\dP(X)>\beta_\alpha+\varepsilon,\
X\in\gamma^{\alpha+\varepsilon}(\aleph_n)\}\cap G$. On $E$, (I) gives
$R_N(X)\ge\delta^-(X)>\beta_\alpha+\varepsilon/2$, and every
$i\in\mathcal{C}(X)$ satisfies
$\delta^+(X_i)\ge R_i(X)\ge R_N(X)>\beta_\alpha+\varepsilon/2$, hence
$\dP(X_i)>\beta_\alpha$ by the contrapositive of (II). Writing $X_N:=X$ and
$Y_N:=\#\{i\le N:\dP(X_i)>\beta_\alpha\}\sim\mathrm{Bin}(N,p_0)$ with
$p_0:=P(\dP(X)>\beta_\alpha)\le\alpha$, the event
$X\in\gamma^{\alpha+\varepsilon}(\aleph_n)$ means
$|\mathcal{C}(X)|>(\alpha+\varepsilon)N$, so $Y_N>(\alpha+\varepsilon)N$
on $E$, and Hoeffding's inequality gives
$P(E)\le P(Y_N>(\alpha+\varepsilon)N)\le e^{-2N(\alpha+\varepsilon-p_0)^2}
\le e^{-2N\varepsilon^2}$.
\end{proof}

\subsection{Hausdorff convergence of the empirical central sets}

Recall the empirical central set
$\hat Q_\beta:=\{z\in\R^d:\delta_{P^{\aleph_n}}(z)\le\beta\}$, $\beta\ge0$,
which by Corollary~\ref{cor:A-empirical-radius} is the sublevel set of the
$\mathbf{k}$-NN distance to the sample, $\mathbf{k}=\lfloor n/2\rfloor+1$;
its computation is discussed in Section~\ref{comp}.

\begin{prop}\label{prop:Qhat-H-conv}
Assume \ref{H1}, \ref{H3} and $\beta_\alpha>0$ (the latter holds
automatically under \ref{H2}, by Remark~\ref{rem:beta-positive}). Then
$d_H(\hat Q_{\beta_\alpha},\Qba)\to0$ almost surely as $n\to\infty$.
\end{prop}

\begin{proof}
Fix $K$ with $K^2\ge d+2$, and set $\theta_n:=K(\log n/n)^{1/2}$,
$\rho_n:=\kappa\theta_n$ and
$G_n:=\{\sup_{A\in\mathcal{A}_d}|P^{\aleph_n}(A)-P(A)|\le\theta_n\}$. By
Theorem~\ref{thm:thluc} (for $n$ large enough that $\theta_n\le1$),
\[
P(G_n^c)\le C\big(n^{2d+2}+1\big)e^{-2n\theta_n^2}
=C\big(n^{2d+2}+1\big)n^{-2K^2}\le2C\,n^{2d+2-2K^2}\le2C\,n^{-2},
\]
which is summable; by Borel--Cantelli, almost surely $G_n$ holds for all $n$
large. Fix such an outcome, and let $n$ also be large enough that
$\theta_n\le u_0$ and $\rho_n\le\min\{\rho_0,\beta_\alpha\}$.

On $G_n$, Lemma~\ref{lem:deterministic-bracketing} (with
$\tilde P=P^{\aleph_n}$, $m=1/2$, $\theta=\theta_n$) gives
$\delta^-_n(y)\le\delta_{P^{\aleph_n}}(y)\le\delta^+_n(y)$ for all $y$,
where $\delta^\pm_n:=\delta_{P,1/2\pm\theta_n}$. We claim that
\begin{equation}\label{eq:sandwich-Qhat}
Q_{\beta_\alpha-\rho_n}\subset\hat Q_{\beta_\alpha}\subset
Q_{\beta_\alpha+\rho_n}.
\end{equation}
For the first inclusion, let $\dP(x)\le\beta_\alpha-\rho_n$ and suppose
$\delta_{P^{\aleph_n}}(x)>\beta_\alpha$. Then
$\delta^+_n(x)\ge\delta_{P^{\aleph_n}}(x)>\beta_\alpha$ and
$\delta^-_n(x)\le\dP(x)\le\beta_\alpha-\rho_n$, so
$\beta_\alpha\in I_{\theta_n}(x)$ and the interval meets the band
$[\beta_\alpha-\rho_0,\beta_\alpha+\rho_0]$; \ref{H1} bounds its length by
$\kappa\theta_n=\rho_n$, whence
$\delta^+_n(x)\le\delta^-_n(x)+\rho_n\le\beta_\alpha$, a contradiction. For
the second inclusion, let $\delta_{P^{\aleph_n}}(x)\le\beta_\alpha$ and
suppose $\dP(x)>\beta_\alpha+\rho_n$. Then
$\delta^+_n(x)\ge\dP(x)>\beta_\alpha+\rho_n$ and
$\delta^-_n(x)\le\delta_{P^{\aleph_n}}(x)\le\beta_\alpha$, so
$\beta_\alpha\in I_{\theta_n}(x)$, the interval meets the band, and \ref{H1}
gives $\delta^+_n(x)\le\beta_\alpha+\rho_n$, again a contradiction.

Now fix $\varepsilon>0$ and let $\lambda>0$ be as in \ref{H3}. For $n$
large, $\rho_n\le\lambda$ and $\beta_\alpha-\rho_n\ge0$, so
$d_H(Q_{\beta_\alpha\pm\rho_n},\Qba)\le\varepsilon$. By
\eqref{eq:sandwich-Qhat}, every $x\in\hat Q_{\beta_\alpha}$ lies in
$Q_{\beta_\alpha+\rho_n}$, so
$d(x,\Qba)\le d_H(Q_{\beta_\alpha+\rho_n},\Qba)\le\varepsilon$; and every
$x\in\Qba$ satisfies
$d(x,\hat Q_{\beta_\alpha})\le d(x,Q_{\beta_\alpha-\rho_n})
\le d_H(Q_{\beta_\alpha-\rho_n},\Qba)\le\varepsilon$. Hence
$d_H(\hat Q_{\beta_\alpha},\Qba)\le\varepsilon$ for all $n$ large, almost
surely.
\end{proof}

\begin{rem}\label{rem:end-to-end-sampling}
Suppose the sample is drawn i.i.d.\ from the contaminated law
$P_\eta=(1-\eta)P_0+\eta Q$ of Section~\ref{contamination}, with
$\eta\le\eta_1$, and that $P_\eta$ itself satisfies \ref{H1}--\ref{H2} at
its own level $\beta^\eta_\alpha$, with constants that can be taken uniform
over $\eta\le\eta_1$ (by Proposition~\ref{prop:contamination-bracketing}
and Lemma~\ref{cor:contamination-local-stability}, the mass-radius
intervals of $P_\eta$ are controlled by those of $P_0$). Then
Corollary~\ref{cor:rate-symmetric-difference}, applied to $P_\eta$, bounds
$P_\eta(X\in\gamma^\alpha\,\triangle\,Q^\eta_{\beta^\eta_\alpha})$, and
since $P_0\le(1-\eta)^{-1}P_\eta$,
\begin{multline*}
P_0\big(X_0\in\gamma^\alpha\,\triangle\,Q^0_{\beta^0_\alpha}\big)
\le\frac{1}{1-\eta}\,
P_\eta\big(X\in\gamma^\alpha\,\triangle\,Q^\eta_{\beta^\eta_\alpha}\big)
\\
+P_0\big(X_0\in Q^\eta_{\beta^\eta_\alpha}\,\triangle\,Q^0_{\beta^0_\alpha}\big)
=O\big((\log n/n)^{1/2}\big)+O(\eta)
\end{multline*}
by Theorem~\ref{cor:calibrated-set-stability}: the conformal region built
from contaminated data approaches the \emph{clean} central set up to an
$O(\eta)$ term, while retaining the distribution-free clean coverage of
Remark~\ref{rem:clean-coverage}.
\end{rem}

\section{Ball representations and a conservative inner proxy of \texorpdfstring{$\hat Q_\beta$}{Q-hat beta}}\label{comp}

Everything in this section uses only the matrix of pairwise distances, so
it applies verbatim on a general metric space; we keep the notation of a
sample $\aleph_n=(x_1,\dots,x_n)$ of points of $\R^d$ (now regarded as
fixed). Recall the empirical central set
$\hat Q_\beta=\{z:\delta_{P^{\aleph_n}}(z)\le\beta\}$ and that, by
Corollary~\ref{cor:A-empirical-radius},
$\delta_{P^{\aleph_n}}(z)=d_{(\mathbf{k})}(z)$, the distance from $z$ to
its $\mathbf{k}$-th nearest sample point, $\mathbf{k}=\lfloor n/2\rfloor+1$.
For $z\in\R^d$ and $\beta\ge0$ set $t(z):=\#\{i\le n:\|z-x_i\|\le\beta\}$,
the number of sample balls of radius $\beta$ containing $z$.

\begin{prop}\label{prop:Qhat-ball-representation}
For every $\beta\ge0$,
\[
\hat Q_\beta=\bigcup_{\substack{I\subset\{1,\dots,n\}\\ |I|=\mathbf{k}}}\
\bigcap_{i\in I}B(x_i,\beta)=\{z:\ t(z)\ge\mathbf{k}\}.
\]
\end{prop}

\begin{proof}
$\delta_{P^{\aleph_n}}(z)=d_{(\mathbf{k})}(z)\le\beta$ holds if and only if
at least $\mathbf{k}$ sample points lie within distance $\beta$ of $z$,
that is, if and only if $t(z)\ge\mathbf{k}$, that is, if and only if $z$
belongs to the intersection of the balls indexed by some $I$ with
$|I|=\mathbf{k}$; see Figure~\ref{fig:Qhat-balls}.
\end{proof}

\begin{figure}[ht]
\centering
\includegraphics[width=0.35\textwidth]{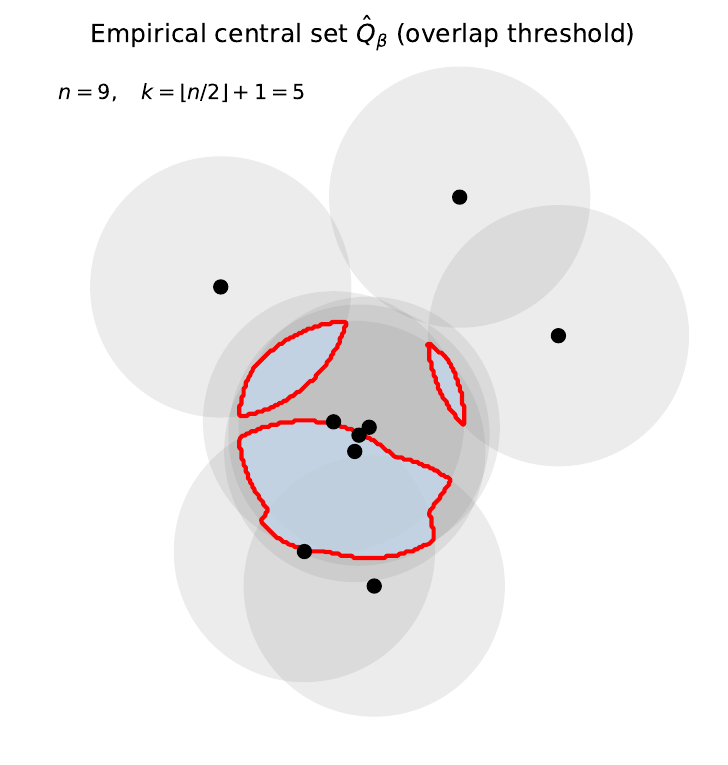}
\caption{The empirical central set $\hat Q_\beta$ as a union of
$\mathbf{k}$-fold intersections of balls of radius $\beta$ centred at the
sample points.}
\label{fig:Qhat-balls}
\end{figure}

There is also a dual representation, through intersections of unions of
balls.

\begin{prop}\label{prop:Sbeta-dual-representation}
Let $\hat S_\beta:=\bigcap_{I\subset\{1,\dots,n\},\,|I|\ge\mathbf{k}}\ 
\bigcup_{i\in I}B(x_i,\beta)$. Then $z\in\hat S_\beta$ if and only if
$t(z)\ge n-\mathbf{k}+1=\lceil n/2\rceil$. Consequently
$\hat Q_\beta\subset\hat S_\beta$, with equality when $n$ is odd; when $n$
is even, $\hat S_\beta\setminus\hat Q_\beta=\{z:\ t(z)=n/2\}$.
\end{prop}

\begin{proof}
Let $J(z):=\{i:\|z-x_i\|\le\beta\}$, so $t(z)=|J(z)|$. A point $z$ fails to
belong to $\bigcup_{i\in I}B(x_i,\beta)$ if and only if $I\subset J(z)^c$;
such an $I$ with $|I|\ge\mathbf{k}$ exists if and only if
$n-t(z)\ge\mathbf{k}$. Hence $z\in\hat S_\beta$ if and only if
$t(z)\ge n-\mathbf{k}+1$. With $\mathbf{k}=\lfloor n/2\rfloor+1$ one has
$n-\mathbf{k}+1=n-\lfloor n/2\rfloor=\lceil n/2\rceil$. If $n$ is odd,
$\lceil n/2\rceil=\lfloor n/2\rfloor+1=\mathbf{k}$ and the two thresholds
coincide; if $n$ is even, $\lceil n/2\rceil=n/2=\mathbf{k}-1$, and the two
sets differ exactly on $\{t=n/2\}$.
\end{proof}

Evaluating $t(z)$ on a fine grid can be costly in moderate dimension. The
following inner approximation is certified and requires only the pairwise
distances of the sample. For $i\le n$ let $\mathcal{D}_i$ be the distance
from $x_i$ to its $(\mathbf{k}-1)$-th nearest neighbour among
$\{x_j\}_{j\ne i}$, well defined for $n\ge2$ since
$1\le\mathbf{k}-1\le n-1$.

\begin{prop}[Conservative proxy]\label{prop:conservative-proxy}
For every $\beta\ge0$,
$\bigcup_{i:\,\mathcal{D}_i<\beta}B(x_i,\beta-\mathcal{D}_i)\subset\hat Q_\beta$.
\end{prop}

\begin{proof}
Let $\mathcal{D}_i<\beta$ and $\|z-x_i\|\le\beta-\mathcal{D}_i$. Each of the
$\mathbf{k}-1$ nearest neighbours $x_j$ of $x_i$ satisfies
$\|z-x_j\|\le\|z-x_i\|+\|x_i-x_j\|\le(\beta-\mathcal{D}_i)+\mathcal{D}_i=\beta$;
together with $x_i$ itself ($\|z-x_i\|\le\beta-\mathcal{D}_i\le\beta$),
this gives $t(z)\ge\mathbf{k}$, so $z\in\hat Q_\beta$ by
Proposition~\ref{prop:Qhat-ball-representation}.
\end{proof}

At the sample points, $x_i\in\hat Q_\beta$ if and only if
$\mathcal{D}_i\le\beta$, and the proxy contains every $x_i$ with
$\mathcal{D}_i<\beta$; it is thus exact at the sample points up to the
boundary case $\mathcal{D}_i=\beta$, and it provides a quick certified
visualization of $\hat Q_\beta$ without any grid evaluation; see
Figure~\ref{fig:proxy}.

\begin{figure}[ht]
\centering
\includegraphics[width=0.3\textwidth]{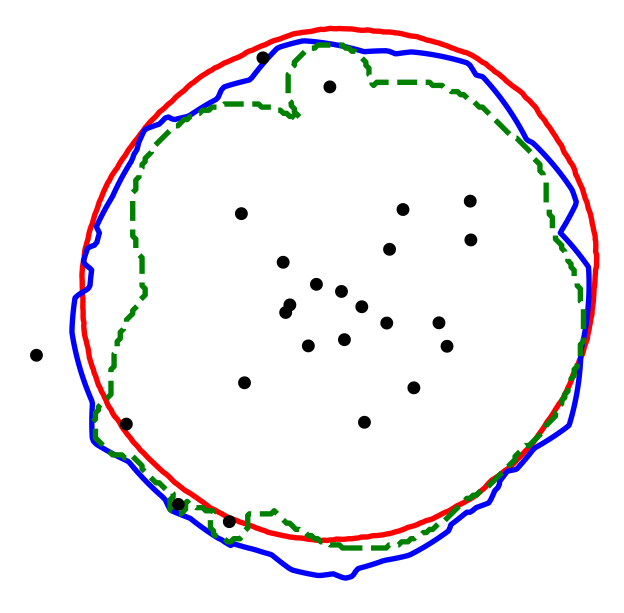}\hfill
\includegraphics[width=0.3\textwidth]{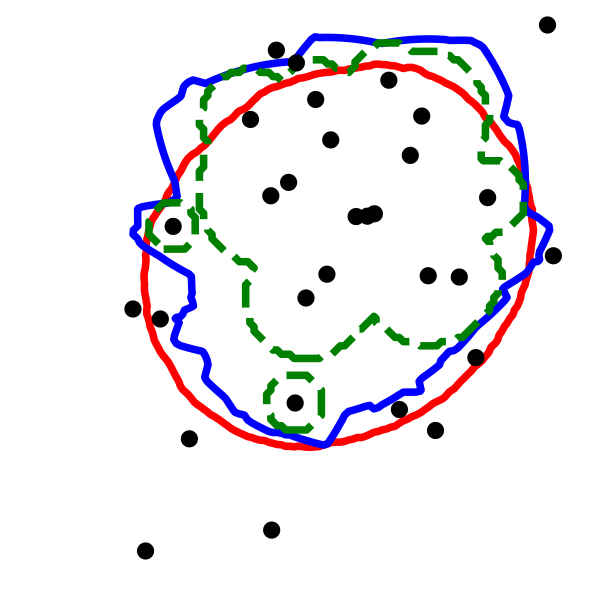}\hfill
\includegraphics[width=0.3\textwidth]{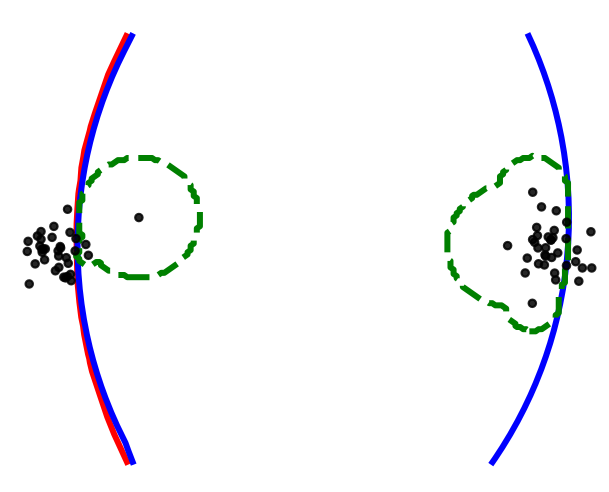}
\caption{Empirical central set $\hat Q_\beta$ (light) and the conservative
inner proxy of Proposition~\ref{prop:conservative-proxy} (dark), for a
moderate level (left), a large $\alpha$ (centre), and a strongly separated
bimodal law for which $\hat Q_\beta$ is disconnected (right).}
\label{fig:proxy}
\end{figure}

\paragraph{Link with the conformal set, and computational cost.}
The conformal region $\gamma^\alpha(\aleph_n)$ is \emph{not} of the form
$\hat Q_\beta$ for a fixed $\beta$: the threshold implicit in the $p$-value
\eqref{eq:pvalue} depends on the query point through the leave-one-out
scores. Theorem~\ref{thm:consistency-corrected-omega} and
Proposition~\ref{prop:Qhat-H-conv} show, however, that both objects target
the same population set $\Qba$, and in practice $\hat Q_{\hat\beta}$, with
$\hat\beta$ an empirical $(1-\alpha)$-quantile of the leave-one-out
half-mass radii, is a convenient surrogate. Regarding cost: once the
pairwise distances of the sample are stored, evaluating the full conformal
$p$-value at a query point $z$ requires the $n$ distances $\|z-x_i\|$ and,
for each $i\le n+1$, the $\mathbf{k}$-th order statistic of the $i$-th row
of the augmented distance matrix; with a linear-time selection algorithm
this is $O(n)$ per row, hence $O(n^2)$ per query point. A split-conformal
variant---computing half-mass radii with respect to one half of the sample
and calibrating the threshold on the other half---reduces the cost to
$O(n)$ per query while retaining finite-sample validity by the standard
split argument and targeting the same population set; its analysis is
simpler than the full conformal one carried out here, and we do not detail
it.

\section{Numerical experiments}\label{experiments}

All experiments were implemented in Python; the code reproducing every
figure and table is available at \texttt{[repository URL]}. Unless stated
otherwise, the base distribution is the bimodal Gaussian mixture
$P_0=\frac12N(m,I_d)+\frac12N(-m,I_d)$, with $m=(2,0,\dots,0)\in\R^d$,
which has a continuous, strictly positive density.

\subsection{Marginal coverage}

We first verify the finite-sample validity of
Proposition~\ref{prop:conformal-validity}. For each configuration of
$(n,d,\alpha)$ we draw $M=2000$ independent replications of
$(\aleph_n,X_{n+1})$ from $P_0$ and report the empirical frequency of the
event $X_{n+1}\in\gamma^\alpha(\aleph_n)$, computed exactly through the
$\mathbf{k}$-NN representation of Corollary~\ref{cor:A-empirical-radius}.
The Monte Carlo standard error is at most $1/(2\sqrt{M})\approx0.011$. The
estimates in Table~\ref{tab:coverage} fluctuate around the nominal level:
despite the ties discussed in Remark~\ref{rem:no-upper-bound}, the regions
are not conservative in practice.

\begin{table}[ht]
\centering
\caption{Empirical marginal coverage of $\gamma^\alpha(\aleph_n)$ at
nominal level $1-\alpha=0.90$, over $M=2000$ replications (Monte Carlo
s.e.\ $\le0.011$). The exact conformal guarantee concerns the true
coverage, so Monte Carlo estimates may fall slightly below $0.90$.}
\label{tab:coverage}
\begin{tabular}{lccccc}
\hline
& $n=50$ & $n=100$ & $n=200$ & $n=500$ & $n=1000$\\
\hline
$d=1$  & 0.909 & 0.886 & 0.898 & 0.905 & 0.907\\
$d=2$  & 0.903 & 0.890 & 0.897 & 0.887 & 0.904\\
$d=5$  & 0.901 & 0.902 & 0.888 & 0.904 & 0.907\\
$d=10$ & 0.905 & 0.912 & 0.894 & 0.900 & 0.902\\
\hline
\end{tabular}
\end{table}

\subsection{Robustness under contamination}

We next illustrate the stability results of Section~\ref{contamination}.
With $d=2$, $n=500$ and $\alpha=0.10$, we sample from
$P_\eta=(1-\eta)P_0+\eta Q$ with $Q=N((10,0),0.25\,I_2)$, for $\eta$
ranging from $0$ to $0.15$.
We compare three scores: the half-mass radius $S_{\mathrm{HM}}$
($\mathbf{k}$-NN distance, $\mathbf{k}=\lfloor n/2\rfloor+1$), the $5$-NN
distance $S_{5\mathrm{NN}}$, and the centred norm
$S_{\mathrm{mean}}(z)=\|z-\bar x_n\|$. For each score $s$, the threshold
$\tau_s$ is calibrated \emph{once on clean data}, so that the region
$\{S_s\le\tau_s\}$ has coverage $1-\alpha$ under $P_0$ (estimated on a
reference sample of size $10^5$); the regions are then recomputed from
contaminated samples and compared with the clean population target through
the Lebesgue measure of the symmetric difference (evaluated on a fine
grid) and through the Hausdorff distance. Keeping $\tau_s$ fixed isolates
the sensitivity of the \emph{score} to contamination, which is the
mechanism quantified by Lemmas~\ref{cor:contamination-local-stability}
and~\ref{cor:contamination-set-stability}; a variant recalibrating the
threshold on the contaminated data, corresponding to
Theorem~\ref{cor:calibrated-set-stability}, behaves similarly and is
provided in the repository.

\begin{figure}[ht]
\centering
\includegraphics[width=0.45\textwidth]{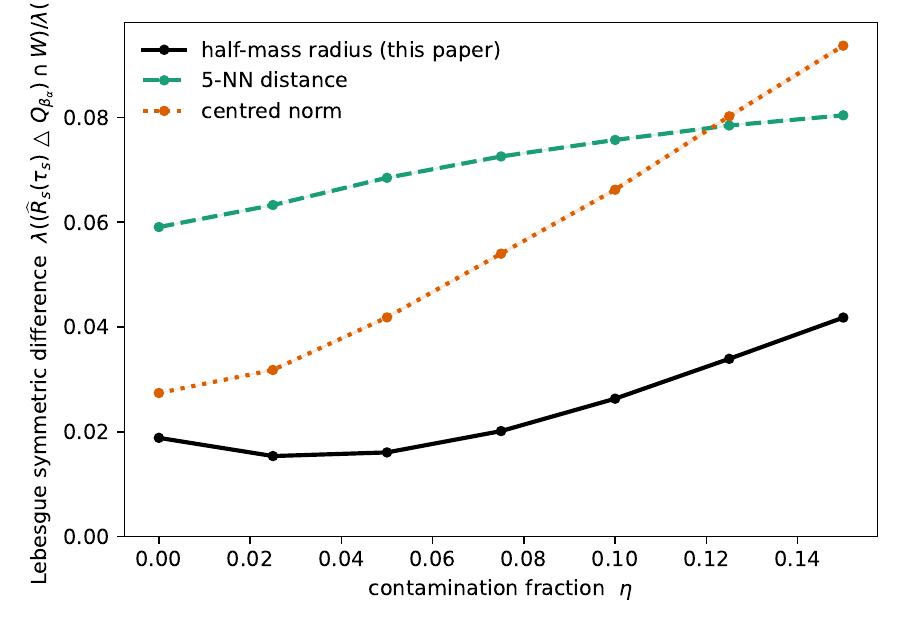}\hfill
\includegraphics[width=0.45\textwidth]{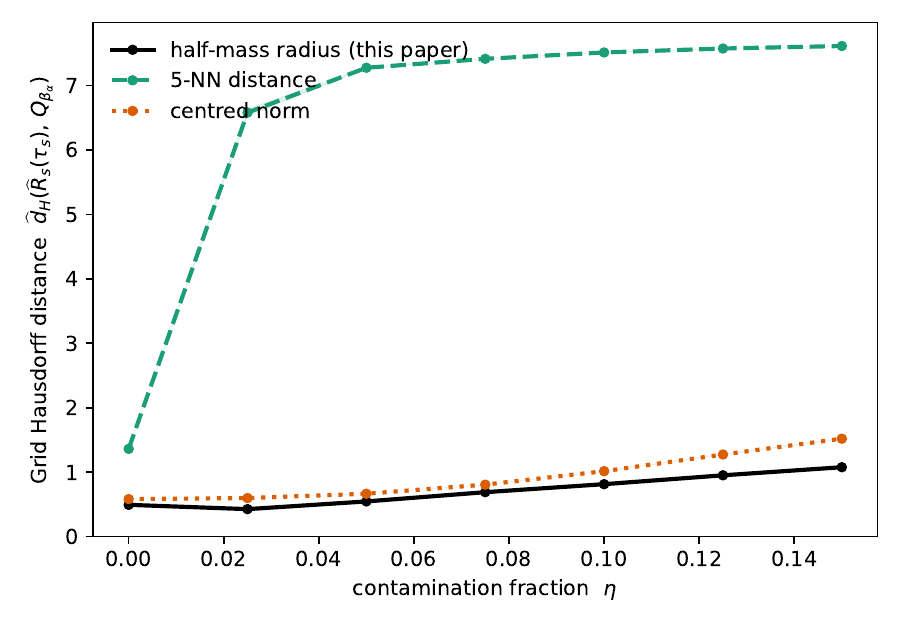}
\caption{Distance from the estimated region to the clean population
target, as a function of the contamination level $\eta$: symmetric
difference (left) and Hausdorff distance (right), for the half-mass
radius, the $5$-NN distance and the centred norm.}
\label{fig:robust-curves}
\end{figure}

\begin{figure}[ht]
\centering
\includegraphics[width=0.3\textwidth]{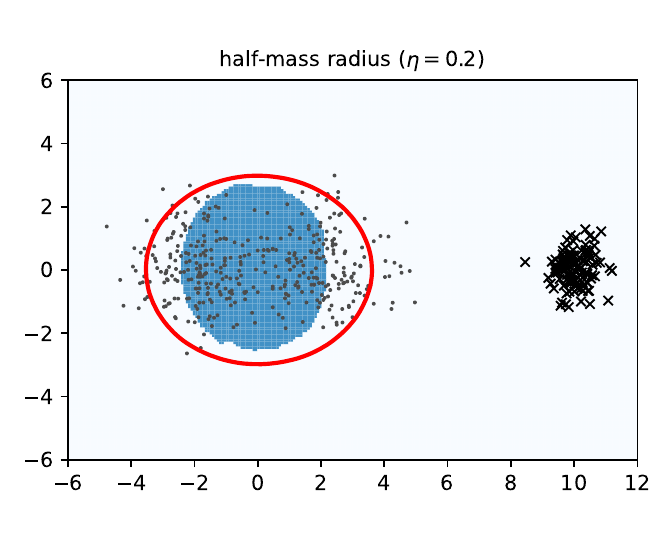}\hfill
\includegraphics[width=0.3\textwidth]{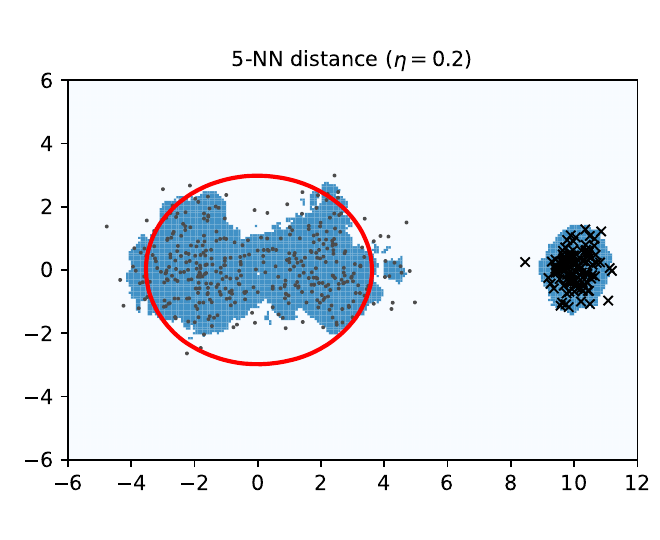}\hfill
\includegraphics[width=0.3\textwidth]{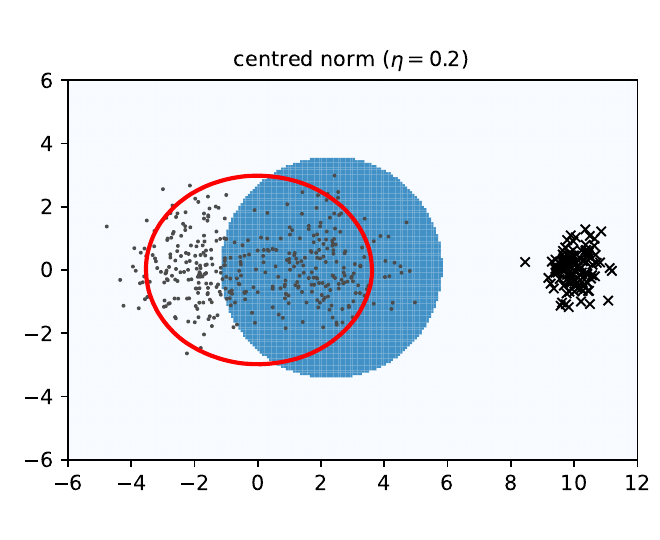}
\caption{Estimated regions under contamination ($d=2$, $n=500$,
$\alpha=0.10$): half-mass radius (left), $5$-NN distance (centre) and
centred norm (right). The half-mass region remains aligned with the clean
central target; the $5$-NN region develops spurious components around the
outliers; the centred-norm region is dragged towards the contamination.}
\label{fig:robust-regions}
\end{figure}

The results, summarized in Figures~\ref{fig:robust-curves}
and~\ref{fig:robust-regions}, match the theory: the degradation of the
half-mass region is small and grows roughly linearly in $\eta$, as
predicted by the $O(\eta)$ bounds of Section~\ref{contamination}; the
centred norm is the most affected in symmetric difference (the whole
region is translated and inflated towards the contamination), while the
$5$-NN score is the most affected in Hausdorff distance (it creates
spurious components around the contaminating cluster).

\subsection{Hausdorff convergence}

Finally we illustrate Proposition~\ref{prop:Qhat-H-conv}. The population
target $\Qba$ is approximated through a reference sample of size $10^6$,
and $d_H(\hat Q_{\beta_\alpha},\Qba)$ is computed for increasing sample
sizes $n$. Figure~\ref{fig:hausdorff-loglog} displays the results in
log-log scale: a least-squares fit gives a slope of approximately $-0.41$,
slightly flatter than the $n^{-1/2}$ reference line.
Proposition~\ref{prop:Qhat-H-conv} asserts no rate, but its proof
sandwiches $\hat Q_{\beta_\alpha}$ between the population sets at levels
$\beta_\alpha\pm\rho_n$, with $\rho_n$ of order $(\log n/n)^{1/2}$, and the
observed decay is consistent with this scale.

\begin{figure}[ht]
\centering
\includegraphics[width=0.4\textwidth]{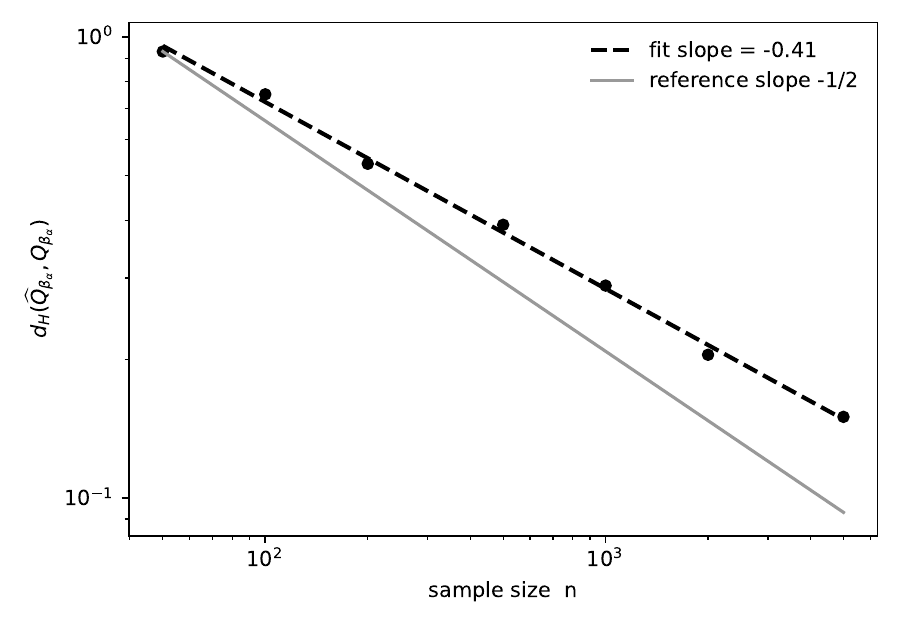}
\caption{Hausdorff distance between $\hat Q_{\beta_\alpha}$ and the
population target $\Qba$ as a function of $n$, in log-log scale, with an
$n^{-1/2}$ reference line. The fitted slope is approximately $-0.41$.}
\label{fig:hausdorff-loglog}
\end{figure}

\subsection{Summary}

The experiments confirm the three pillars of the theory: exact
finite-sample marginal validity of the conformal region
(Table~\ref{tab:coverage}); robustness of the half-mass score under
contamination, with a degradation that is linear in $\eta$ and markedly
smaller than that of non-robust scores
(Figures~\ref{fig:robust-curves}--\ref{fig:robust-regions}); and
convergence of the empirical central sets to the population target, at a
rate compatible with $(\log n/n)^{1/2}$ (Figure~\ref{fig:hausdorff-loglog}).

\appendix
\section{Technical results}\label{app}

\begin{lem}\label{lem:Qbeta-compact}
For every $\beta\ge0$ the set $Q_\beta$ is compact.
\end{lem}

\begin{proof}
$Q_\beta$ is closed by continuity of $\dP$. For boundedness, let
$K\subset\R^d$ be a compact set with $P(K)>1/2$ (such $K$ exists since
every Borel probability on $\R^d$ is tight). If $r<d(x,K)$ then
$B(x,r)\cap K=\emptyset$, so $P(B(x,r))\le1-P(K)<1/2$; hence
$\dP(x)\ge d(x,K)$ for every $x$, and
$Q_\beta\subset\{x:d(x,K)\le\beta\}$, which is bounded.
\end{proof}

The following proposition gives a simple topological characterization of
the regularity assumption \ref{H3}.

\begin{prop}\label{prop:H3-characterization}
Assume $\beta_\alpha>0$. Then \ref{H3} holds if and only if
\begin{equation}\label{eq:no-flat}
\Qba=\overline{\{x\in\R^d:\dP(x)<\beta_\alpha\}} .
\end{equation}
\end{prop}

\begin{proof}
Note first that $\overline{\{\dP<\beta_\alpha\}}\subset\Qba$ always, by
continuity of $\dP$, and that $\Qba\ne\emptyset$ since
$P(\Qba)\ge1-\alpha>0$.

Assume \eqref{eq:no-flat} and fix $\varepsilon>0$. We claim that there is
$\lambda_1>0$ such that $\beta_\alpha\le\beta\le\beta_\alpha+\lambda_1$
implies $\sup_{x\in Q_\beta}d(x,\Qba)\le\varepsilon$; since
$\Qba\subset Q_\beta$, this bounds $d_H(Q_\beta,\Qba)$. If not, there are
$\beta_n\downarrow\beta_\alpha$ and $x_n\in Q_{\beta_n}$ with
$d(x_n,\Qba)>\varepsilon$. The $x_n$ eventually lie in the compact set
$Q_{\beta_\alpha+1}$ (Lemma~\ref{lem:Qbeta-compact}), so along a
subsequence $x_n\to x$, and by continuity
$\dP(x)\le\lim_n\beta_n=\beta_\alpha$, that is $x\in\Qba$, contradicting
$d(x_n,\Qba)>\varepsilon$. For levels below $\beta_\alpha$, by
\eqref{eq:no-flat} every $x\in\Qba$ lies within $\varepsilon/2$ of some $y$
with $\dP(y)<\beta_\alpha$. The open balls
$\{B^{\circ}(y,\varepsilon/2):\dP(y)<\beta_\alpha\}$ therefore cover the
compact set $\Qba$; extract a finite subcover with centres
$y_1,\dots,y_m$ and set $\beta^*:=\max_{j\le m}\dP(y_j)<\beta_\alpha$. For
every $\beta\in(\beta^*,\beta_\alpha]$ we then have $y_j\in Q_\beta$ for
all $j$, so every $x\in\Qba$ is within $\varepsilon/2$ of $Q_\beta$; since
also $Q_\beta\subset\Qba$, we get $d_H(Q_\beta,\Qba)\le\varepsilon/2$ (and
in particular $Q_\beta\ne\emptyset$). Taking
$\lambda:=\min\{\lambda_1,\beta_\alpha-\beta^*\}>0$ proves \ref{H3}; note
that $\beta_\alpha>0$ guarantees that levels
$\beta\in(\beta^*,\beta_\alpha)$ with $\beta\ge0$ exist.

Suppose \eqref{eq:no-flat} fails, and pick
$x_0\in\Qba\setminus\overline{\{\dP<\beta_\alpha\}}$. If
$\{\dP<\beta_\alpha\}=\emptyset$, then $Q_\beta=\emptyset$ for every
$\beta<\beta_\alpha$ and $d_H(Q_\beta,\Qba)=+\infty$, so \ref{H3} fails.
Otherwise let
$\varepsilon_0:=d(x_0,\overline{\{\dP<\beta_\alpha\}})>0$. For every
$\beta<\beta_\alpha$, $Q_\beta\subset\{\dP<\beta_\alpha\}$, hence
$d(x_0,Q_\beta)\ge\varepsilon_0$ and $d_H(Q_\beta,\Qba)\ge\varepsilon_0$,
so again \ref{H3} fails.
\end{proof}

Condition \eqref{eq:no-flat} rules out ``flat parts'' of $\dP$ at the level
$\beta_\alpha$: the sublevel set must be the closure of the strict
sublevel set. It holds, for instance, whenever $\dP$ has no local minima at
height exactly $\beta_\alpha$ in regions of positive reach, and it is
straightforward to verify numerically.

The next lemma is used in the proofs of Lemma~\ref{lem:quantile-stability}
and Theorem~\ref{cor:calibrated-set-stability}; the one after it quantifies
the fixed-threshold scenario discussed at the end of
Section~\ref{contamination}. Recall that $\rho_\eta=\kappa_0u_\eta$.

\begin{lem}\label{cor:contamination-local-stability}
Assume that $P_0$ satisfies \eqref{eq:local-strip-condition-contamination}
at a level $\beta^*>0$, and let $0\le\eta<1/2$ be such that
$u_\eta\le u_0$. Then: (i) for every $x$ with
$I^0_{u_\eta}(x)\cap[\beta^*-\rho_0,\beta^*+\rho_0]\ne\emptyset$,
$|\delta_\eta(x)-\delta_0(x)|\le\rho_\eta$; (ii) for every $\beta>0$ with
$|\beta-\beta^*|\le\rho_0$,
$Q^0_{\beta-\rho_\eta}\subset Q^\eta_\beta\subset Q^0_{\beta+\rho_\eta}$.
\end{lem}

\begin{proof}
(i) By Proposition~\ref{prop:contamination-bracketing} and
$a_\eta=1/2-u_\eta$, $b_\eta=1/2+u_\eta$, both $\delta_\eta(x)$ and
$\delta_0(x)$ belong to $I^0_{u_\eta}(x)$, an interval whose length is at
most $\kappa_0u_\eta=\rho_\eta$ by
\eqref{eq:local-strip-condition-contamination}.

(ii) For the first inclusion, let $\delta_0(x)\le\beta-\rho_\eta$ and
suppose, for contradiction, that $\delta_\eta(x)>\beta$. Then
$\delta_{P_0,1/2-u_\eta}(x)\le\delta_0(x)\le\beta-\rho_\eta<\beta
<\delta_\eta(x)\le\delta_{P_0,1/2+u_\eta}(x)$, so
$\beta\in I^0_{u_\eta}(x)$, and since $|\beta-\beta^*|\le\rho_0$ the
interval meets the band. By \eqref{eq:local-strip-condition-contamination}
its length is at most $\rho_\eta$, whence
$\delta_\eta(x)\le\delta_{P_0,1/2+u_\eta}(x)\le\delta_{P_0,1/2-u_\eta}(x)+\rho_\eta\le\beta$,
a contradiction. For the second inclusion, let $\delta_\eta(x)\le\beta$
and suppose $\delta_0(x)>\beta+\rho_\eta$. Then
$\delta_{P_0,1/2-u_\eta}(x)\le\delta_\eta(x)\le\beta$ and
$\delta_{P_0,1/2+u_\eta}(x)\ge\delta_0(x)>\beta+\rho_\eta$, so again
$\beta\in I^0_{u_\eta}(x)$ and the interval meets the band; its length then
exceeds $\rho_\eta$, contradicting
\eqref{eq:local-strip-condition-contamination}.
\end{proof}

\begin{lem}\label{cor:contamination-set-stability}
Let $X_0\sim P_0$. Assume that $P_0$ satisfies
\eqref{eq:local-strip-condition-contamination} at $\beta^*=\beta^0_\alpha$
and Assumption~\ref{H2} at $\beta^0_\alpha$ (with constants
$L,c_-,c_+,t_0$). If $u_\eta\le u_0$ and $\rho_\eta\le t_0$, then
\[
P\big(X_0\in Q^\eta_{\beta^0_\alpha}\,\triangle\,Q^0_{\beta^0_\alpha}\big)
\le L\rho_\eta\le L\kappa_0\eta .
\]
\end{lem}

\begin{proof}
By Remark~\ref{rem:beta-positive}, \ref{H2} forces $\beta^0_\alpha>0$, so
Lemma~\ref{cor:contamination-local-stability}(ii) applies at
$\beta=\beta^0_\alpha$ and gives
$Q^0_{\beta^0_\alpha-\rho_\eta}\subset Q^\eta_{\beta^0_\alpha}\subset
Q^0_{\beta^0_\alpha+\rho_\eta}$. Hence
$Q^\eta_{\beta^0_\alpha}\triangle Q^0_{\beta^0_\alpha}\subset\{x:
\beta^0_\alpha-\rho_\eta<\delta_0(x)\le\beta^0_\alpha+\rho_\eta\}$, whose
$P_0$-mass is at most $L\rho_\eta$ by the upper bound in \ref{H2}, applied
with $t=\rho_\eta\le t_0$.
\end{proof}

\begin{rem}\label{rem:clean-coverage}
No assumptions at all are needed for the following transfer. Let
$X_1,\dots,X_n,X$ be i.i.d.\ with law $P_\eta$, let $X_0\sim P_0$ be
independent of $\aleph_n=(X_1,\dots,X_n)$, and let $\gamma^\alpha(\aleph_n)$
be the conformal region built from $\aleph_n$. Since $P_\eta\ge(1-\eta)P_0$,
Proposition~\ref{prop:conformal-validity} gives
$(1-\eta)P(X_0\notin\gamma^\alpha(\aleph_n))\le
P(X\notin\gamma^\alpha(\aleph_n))\le\alpha$, that is,
$P(X_0\in\gamma^\alpha(\aleph_n))\ge1-\alpha/(1-\eta)$. Thus the
miscoverage level with respect to the clean law is inflated at most by the
factor $(1-\eta)^{-1}$, uniformly over the contamination $Q$.
\end{rem}

\bibliographystyle{plainnat}
\bibliography{bibli}

\end{document}